\documentclass[reqno]{amsart}
\usepackage{amssymb}
\usepackage{amsmath}
\usepackage[mathscr]{euscript}

\usepackage[small]{caption}
\usepackage{psfrag}
\usepackage{graphicx}
\usepackage{color}

\makeatletter
\@addtoreset{equation}{section}
\makeatother

\newtheorem{theorem}{Theorem}[section]
\newtheorem{lemma}[theorem]{Lemma}
\newtheorem{proposition}[theorem]{Proposition}
\newtheorem{corollary}[theorem]{Corollary}

\newtheorem{remark}[theorem]{Remark}

\newcommand{\mc}[1]{{\mathcal #1}}
\newcommand{\mf}[1]{{\mathfrak #1}}
\newcommand{\mb}[1]{{\mathbf #1}}
\newcommand{\bb}[1]{{\mathbb #1}}
\newcommand{\bs}[1]{{\boldsymbol #1}}
\newcommand{\ms}[1]{{\mathscr #1}}

\newcounter{as}[section]

\newtheorem{asser}[as]{Assertion}

\renewcommand{\Cap}{{\rm cap}}

\begin{document}

\title[]{Metastability of reversible random walks in potential fields}

\author{C. Landim, R. Misturini, K. Tsunoda}

\address{\noindent IMPA, Estrada Dona Castorina 110, CEP 22460 Rio de
  Janeiro, Brasil and CNRS UMR 6085, Universit\'e de Rouen, Avenue de
  l'Universit\'e, BP.12, Technop\^ole du Madril\-let, F76801
  Saint-\'Etienne-du-Rouvray, France.  \newline e-mail: \rm
  \texttt{landim@impa.br} }

\address{\noindent IMPA, Estrada Dona Castorina 110, CEP 22460 Rio de
  Janeiro, Brasil.  \newline e-mail: \rm \texttt{misturini@impa.br} }

\address{\noindent Graduate School of Mathematical Sciences, The
  University of Tokyo, Komaba, Tokyo 153-8914, Japan.  \newline e-mail:
  \rm \texttt{tsunoda@ms.u-tokyo.ac.jp}}


\begin{abstract}
  Let $\Xi$ be an open and bounded subset of $\bb R^d$, and let
  $F:\Xi\to\bb R$ be a twice continuously differentiable function.
  Denote by $\Xi_N$ the discretization of $\Xi$, $\Xi_N = \Xi \cap
  (N^{-1} \bb Z^d)$, and denote by $X_N(t)$ the continuous-time,
  nearest-neighbor, random walk on $\Xi_N$ which jumps from $\bs x$ to
  $\bs y$ at rate $ e^{-(1/2) N [F(\bs y) - F(\bs x)]}$. We examine in
  this article the metastable behavior of $X_N(t)$ among the wells of
  the potential $F$.
\end{abstract}

\maketitle

\section{Introduction}
\label{vsec04}

We introduced recently in \cite{bl2, bl7} an approach to prove the
metastable behavior of Markov chains which has been successfully
applied in several different contexts. We refer to \cite{bl9, l2} for
a description of the method and for examples of Markov chains whose
metastable behavior has been established with this approach.

We examine in this article the metastable behavior of reversible
random walks in force fields. This is an old problem whose origin can
be traced back at least to Kramers \cite{kra1}. It has been adressed
by Freidlin and Wentsell \cite{fw1} and by Galves, Olivieri and Vares
\cite{gov1} in the context of small random perturbations of dynamical
systems, and, more recently, by Bovier, Eckhoff, Gayrard and Klein in
a series of papers \cite{begk1, begk2, begk3, bgk1} through the
potential theoretic approach. This problem has raised interest and has
found applications in many areas, as computer sciences \cite{ce1} and
chemical physics \cite{nwpp1}.

The first main result of this article, Theorem \ref{vs07}, states that
starting from a neighborhood of a local minimum of the force field, in
an appropriate time-scale, the evolution of the random walk can be
described by a reversible Markov chain in a finite graph, in which the
vertices represent the wells of the force field and the edges the
saddle points.

More precisely, denote by $X_N(t)$ a reversible random walk evolving
in a discretization of a bounded domain $\Xi\subset \bb R^d$ according
to a force field $F:\Xi \to \bb R$. A precise definition of the
dynamics is given below in \eqref{v54}. Let $\bs x_1, \dots, \bs x_L$
be the local minima of the field $F$, and let $Y_N(t)$ be the process
which records the minima visited: $Y_N(t)$ is equal to $j$ if the
chain $X_N(t)$ belongs to a neighborhood of $\bs x_j$, $1\le j\le L$,
and $0$ otherwise. Clearly, $Y_N(t)$ is not Markovian. Theorem
\ref{vs07} asserts that starting from a neighborhood of a local
minimum $\bs x_j$, there exists a time scale $\beta_N$, which depends
on $j$, in which $Y_N(t\beta_N)$ converges in some topology to a
Markovian dynamics whose state space is a subset of $\{1, \dots,
L\}$. This asymptotic dynamics may have absorbing points, and its jump
rates depend solely on the behavior of the potential in the
neighborhoods of the local minima and in the neighborhoods of the
saddle points. Theorem \ref{vs07} is similar in spirit to the one of
No\'e, Wu, Prinz and Plattner \cite{nwpp1}, who proved that projected
metastable Markovian dynamics can be well approximated by hidden
Markovian dynamics.

The second main result, Theorem \ref{vs20}, adresses the problem of
the exit points from a domain. Consider a local minimum $\bs x_j$ of
the force field and denote by $\{\bs z_1, \dots, \bs z_K\}$ the lowest
saddle points of $F$ which separate $\bs x_j$ from the other local
minima. Theorem \ref{vs20} provides the asymptotic probabilities that
the chain $X_N(t)$ will traverse a mesoscopic neighborhood of a saddle
point $\bs z_i$ before hitting another local minima of the force
field.

We explained already in \cite{bl9} the main differences between our
approach and the potential theoretic one \cite{begk1, begk2}, and
between our approach and the pathwise one due to Cassandro, Galves,
Olivieri and Vares \cite{cgov1}. We will not repeat this exposition
here.  Our approach does not aim to characterize the typical paths in
a transition between two metastable states, in contrast with the
transition path theory \cite{ee1}. Nevertheless, in the case where the
number of wells is small, as in the examples presented in \cite{mse1},
Theorems \ref{vs07} and \ref{vs20} describe the distribution of the
transition paths, at least at the scale of the metastable sets, by
indicating the sequence of metastable sets visited in a transition
between two metastable sets.

In the case of complex networks, the Lennard--Jones clusters analyzed
in \cite{ce1} for instance, to give a rough view of the transition
paths from two metastable states, we may proceed in two ways. One
possibility is to reduce the number of nodes by considering the trace
of the original chain on a subset of the state space
(cf. \cite[Section 6.1]{bl2} for the definition of trace
processes). Avena and Gaudilli\`ere \cite{ag1} proposed a natural
algorithm to reduce the number of vertices of a chain.  The algorithm
produces a subset $V$ with the property that the mean hitting time of
$V$ does not depend on the starting point. In this sense the vertices
of $V$ are ``uniformly'' distributed among the set of nodes. The
algorithm can also be calibrated to provide a large or small set of
nodes $V$.

Another possibility is to identify certain nodes, losing the Markov
property, and to apply Theorem \ref{vs07} below to approximate this
new dynamics by a Markovian dynamics.  To describe the transition
paths at this level of accuracy, one can compute for these reduced
dynamics the equilibrium potential between two metastable sets (the
committor in the terminology of \cite{ce1}), and the optimal flow for
Thomson's principle (the probability current of reactive
trajectories).

In both cases, the selection of the set of nodes or the selection of
nodes to be merged have to be carried out judiciously, to reduce as
much as possible the number of nodes without losing the essential
features of the original chain.  From a computational point of view,
the jump rates of trace process are easily calculated, while the jump
rates of projected processes are more difficult to derive. In the
first case, it suffices to apply recursively the first displayed
equation below the proof of Corollary 6.2 in \cite{bl2}, while in the
second case, one has to calculate the capacities between the
metastable sets.

\section{Notation and Results}
\label{vsec00}

Let $\Xi$ be an open and bounded subset of $\bb R^d$, and denote by
$\partial \,\Xi$ its boundary, which is assumed to be a smooth
manifold. Fix a twice continuously differentiable function
$F:\Xi\cup \partial \,\Xi\to\bb R$, with a finite number of critical
points, satisfying the following assumptions:

\begin{enumerate}
\item[(H1)] The second partial derivatives of $F$ are Lipschitz
  continuous. Denote by $C_1$ the Lipschitz constant;
\item[(H2)] All the eigenvalues of the Hessian of $F$ at the critical
  points which are local minima are \emph{strictly} positive.
\item[(H3)] The Hessian of $F$ at the critical points which are not
  local minima or local maxima has one strictly negative eigenvalue,
  all the other ones being strictly positive. In dimension $1$ this
  assumption requires the second derivative of $F$ at the local minima
  to be strictly negative.
\item[(H4)] For every $\bs x\in \partial\, \Xi$, $(\nabla F)(\bs x)
  \cdot \bs n (\bs x) <0$, where $\bs n (\bs x)$ represents the
  exterior normal to the boundary of $\Xi$, and $\bs x\cdot\bs y$ the
  scalar product of $\bs x$\, $\bs y\in\bb R^d$. This hypothesis
  guarantees that $F$ has no local minima at the boundary of $\Xi$.
\end{enumerate}

Denote by $\Xi_N$ the discretization of $\Xi$: $\Xi_N = \Xi \cap
(N^{-1} \bb Z^d)$, $N\ge 1$, where $N^{-1} \bb Z^d = \{\bs k/N : \bs
k\in \bb Z^d\}$. The elements of $\Xi_N$ are represented by the
symbols $\bs x=(\bs x_1, \dots, \bs x_d)$, $\bs y$ and $\bs z$.  Let
$\mu_N$ be the probability measure on $\Xi_N$ defined by
\begin{equation*}
\mu_N(\bs x) \;=\; \frac 1{Z_N} e^{-N F(\bs x)}\;,\quad\bs x\in
\Xi_N\;, 
\end{equation*}
where $Z_N$ is the partition function $Z_N = \sum_{\bs x\in \Xi_N}
\exp\{-N F(\bs x)\}$.  Let $\{X_N(t) : t\ge 0\}$ be the continuous-time
Markov chain on $\Xi_N$ whose generator $L_N$ is given by
\begin{equation}
\label{v54}
(L_N f)(\bs x) \;=\; \sum_{\substack{\bs y \in\Xi_N\\
\Vert\bs y - \bs x \Vert = 1/N}}  e^{-(1/2) N [F(\bs y) - F(\bs x)]} 
\, [f(\bs y) - f(\bs x)]\;,
\end{equation}
where $\Vert \,\cdot\,\Vert$ represents the Euclidean norm of $\bb
R^d$.  The rates were chosen for the measure $\mu_N$ to be reversible
for the dynamics. Denote by $R_N(\bs x, \bs y)$, $\lambda_N(\bs x)$,
$\bs x$, $\bs y\in\Xi_N$, the jump rates, holding rates of the
chain $X_N(t)$, respectively:
\begin{equation*}
\begin{split}
& R_N(\bs x, \bs y) \;=\;
\begin{cases}
e^{-(1/2) N [F(\bs y) - F(\bs x)]} & \Vert\bs y - \bs x \Vert = 1/N\;,\,
\bs x \,,\, \bs y \in \Xi_N\;,
\\
0 & \text{otherwise}\;.
\end{cases}
\\
& \qquad \lambda_N(\bs x) \;=\; \sum_{\substack{\bs y \in\Xi_N\\
\Vert\bs y - \bs x \Vert = 1/N}}  R_N(\bs x, \bs y) \;.
\end{split}
\end{equation*}

Denote by $D(\bb R_+, \Xi_N)$ the space of right-continuous
trajectories $f: \bb R_+ \to \Xi_N$ with left-limits, endowed with
the Skorohod topology. Let $\mb P_{\bs x}= \mb P^N_{\bs x}$, $\bs x\in
\Xi_N$, be the measure on $D(\bb R_+, \Xi_N)$ induced by the
chain $X_N(t)$ starting from $\bs x$. Expectation with respect to $\mb
P_{\bs x}$ is denoted by $\mb E_{\bs x}$.

For a subset $A$ of $\Xi_N$, denote by $H_A$ (resp. $H^+_A$) the
hitting time of (resp. return time to) the set $A$:
\begin{equation*}
\begin{split}
& H_{A} \;:=\;\inf\{ t> 0 : X_N(t) \in A \}\;, \\
& \quad H^+_A \,:=\, \inf\{ t>0 : X_N(t) \in A \,,\,
X_N(s) \not= X_N(0) \;\;\textrm{for some $0< s < t$}\}\,.
\end{split}
\end{equation*}
The capacity between two disjoint sets $A$, $B$ of $\Xi_N$, denoted
by $\Cap_N(A,B)$, is given by
\begin{equation*}
\Cap_N(A,B) \;=\; \sum_{\bs x\in A} \mu_N(\bs x) \, \lambda_N(\bs x)
\,  \mb P_{\bs x} \big[ H_B < H^+_A \big]\;.
\end{equation*}

\smallskip\noindent{\bf A. The wells and their capacities.}  Denote by
$\mf M$ the set of local minima and by $\mf S$ the set of saddle
points of $F$ in $\Xi$.  Let $\mf S_1$ be the set of the lowest saddle
points:
\begin{equation*}
\mf S_1 \;=\; \Big\{ \bs z\in \mf S : F (\bs z) = \min\{ F (\bs y) :
\bs y\in \mf S\}\, \Big\}\;.
\end{equation*}
We represent by $\bs z^{1,1}, \dots , \bs z^{1, n_1}$ the elements of
$\mf S_1$, $\mf S_1 = \{ \bs z^{1,1}, \dots , \bs z^{1,n_1}
\}$. Starting from $\mf S_1$, we define inductively a finite sequence
of disjoint subsets of $\mf S$. Assume that $\mf S_1, \dots, \mf S_i$
have been defined, let $\mf S^+_i = \mf S_1 \cup \cdots \cup \mf S_i$,
and let
\begin{equation*}
\mf S_{i+1} \;=\; \Big\{ \bs z\in \mf S : F (\bs z) = \min\{ F (\bs y) :
\bs y\in \mf S \setminus \mf S^+_i \}\, \Big\}\;.
\end{equation*}
We denote by $\bs z^{i,j}$, $1\le j\le n_i$ the elements of $\mf S_i$.
We obtain in this way a partition $\{\mf S_i : 1\le i\le i_0\}$ of
$\mf S$.

We will refer to the index $i$ as the level of a saddle point.  Denote
by $H_i$ the height of the saddle points in $\mf S_i$:
\begin{equation*}
H_i\;=\; F(\bs z^{i,1}) \;, \quad 1\le i\le i_0\;,
\end{equation*}
so that $H_1< H_2 < \dots < H_{i_0}$.

For each $1\le i\le i_0$, let $\widehat{\varOmega}^i$ be the subset of
$\Xi$ defined by
\begin{equation*}
\widehat{\varOmega}^i \;=\; \big\{\bs x\in \Xi : F(\bs x) 
\le F(\bs z^{i,1}) \big\} \;.
\end{equation*}
By definition, $\widehat{\varOmega}^i \subset
\widehat{\varOmega}^{i+1}$.  The set $\widehat{\varOmega}^i$ can be
written as a disjoint union of connected components:
$\widehat{\varOmega}^i = \cup_{1\le j\le \varkappa_i}
\widehat{\varOmega}^i_j$, where $\widehat{\varOmega}^i_j \cap
\widehat{\varOmega}^i_k = \varnothing$, $j\not = k$, and where each
set $\widehat{\varOmega}^i_j$ is connected. Some connected component
may not contain any saddle point in $\mf S_i$, and some may contain more
than one saddle point. Denote by $\varOmega^i_j$, $1\le j\le \ell_i$,
the connected components $\widehat{\varOmega}^i_{j'}$ which contain a
point in $\mf S_i$, and let $\varOmega^i = \cup_{1\le j\le \ell_i}
\varOmega^i_j$.  Clearly, the number of components $\varOmega^i_j$ is
smaller than the number of elements of $\mf S_i$, $\ell_i\le n_i$.

Each component $\varOmega^i_j$ is a union of wells, $\varOmega^i_j =
W^i_{j,1} \cup \cdots \cup W^i_{j,\ell^i_j}$. The sets $W^i_{j,a}$ are
defined as follows. Let $\mathring{\varOmega}^i_j$ be the interior of
$\varOmega^i_j$. Each set $W^i_{j,a}$ is the closure of a connected
component of $\mathring{\varOmega}^i_j$.  The intersection of two
wells is a subset of the set of saddle points: $W^i_{j,a} \cap
W^i_{j,b} \subset \mf S_i$. Figure \ref{fig1} illustrates the wells of
two connected components of some level. The sets $W^\epsilon_a$ are
introduced just before \eqref{v16}.

\begin{figure}[htb]
  \centering
  \def\svgwidth{350pt}
  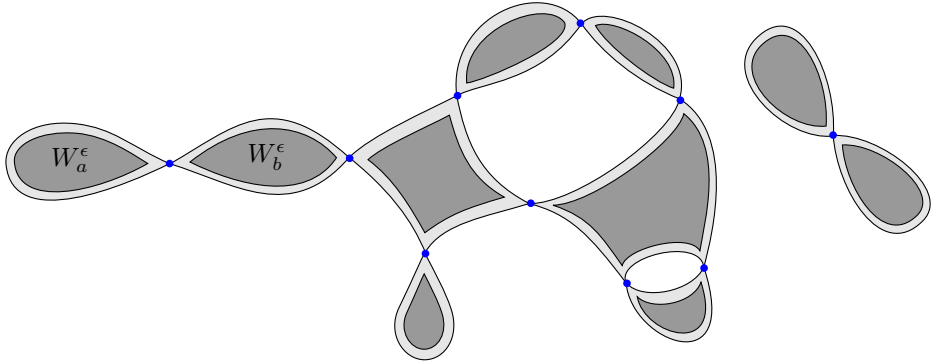
  \caption{Some wells which form two connected components $\Omega^i_1$
    and $\Omega^i_2$.}
\label{fig1}
\end{figure}

Fix $1\le i \le i_0$ and $1\le j\le \ell_i$ and a connected component
$\varOmega = \varOmega^i_j$. To avoid heavy notation, unless when
strictly required, we omit from now on the dependence of the sets $\mf
S_i$, $W^i_{j,a}$ and the numbers $\ell^i_j$ on the indices $i$ and $j$
which are fixed.

Let $S = \{1, \dots, \ell\}$ denote the set of the indices of the
wells forming the connected component $\varOmega$.  For $a\not = b\in
S$, denote by $\mf S_{a,b}$ the set of saddle points separating $W_a$
from $W_b$,
\begin{equation*}
\mf S_{a,b} \;=\; \big\{\bs z \in \mf S : \bs z\in W_a \cap W_b\big\}\;,
\end{equation*}
and denote by $\mf S(A)$, $A\subset S$, the set of saddle points
separating $\cup_{a\in A} W_a$ from $\cup_{a\in A^c} W_a$:
\begin{equation*}
\mf S(A) \;=\; \bigcup_{a \in A \,,\, b \in A^c} \mf S_{a,b} \;.
\end{equation*}
For a saddle point $\bs z \in \mf S$, denote by $- \mu(\bs z)$ the
unique negative eigenvalue of the Hessian of $F$ at $\bs z$.

Recall that $F(\bs z) = H_i$, $\bs z\in \mf S=\mf S_i$.  For
$0<\epsilon< H_i - H_{i-1}$, $1\le a\le \ell$, let $W^\epsilon_a =
\{\bs x \in W_a : F(\bs x) < H_i - \epsilon\}$, and let
\begin{equation}
\label{v16}
\ms E^a_N \;=\; W^\epsilon_a \,\cap\, \Xi_N\;,\;\; 1\le a\le
\ell\,, \quad \ms E_N (A) = \bigcup_{a\in A} \ms E^a_N \;,\;\;
A\subset S\;.
\end{equation}
Each well $W^1_{j,a}$ contains exactly one local minimum of $F$, while
the wells $W^i_{j,a}$, $1<i\le i_0$, may contain more than one local
minimum.  Denote by $\{\bs m_{a,1}, \dots, \bs m_{a,q}\}$, $q=q_{a}$,
the deepest local minima of $F$ which belong to $W^i_{j,a}$:
\begin{equation*}
\{\bs m_{a,1}, \dots, \bs m_{a,q}\}  \;=\; \Big\{\bs y\in W_a \cap \mf M : 
F(\bs y) = \min\{F(\bs y') : \bs y'\in  W_{a} \cap \mf M\}\,\Big\}\;.
\end{equation*}
Let $h_a = F(\bs m_{a,1})$ and let
\begin{equation*}
\bs \mu (a) \;=\; \sum_{k=1}^{q_{a}} 
\frac{1}{\sqrt{\det {\rm Hess}\, F(\bs m_{a, k})}} \;,\;\; a\in S\;,
\end{equation*}
where ${\rm Hess}\, F(\bs x)$ represents the Hessian of $F$ calculated
at $\bs x$, and $\det {\rm Hess}\, F(\bs x)$ its determinant.  A
calculation, presented in \eqref{v06}, shows that for each $a\in S$,
\begin{equation}
\label{v24}
\mu_N(\ms E^a_N) \;=\; [1+o_N(1)]\, \frac{(2\pi N)^{d/2}}{Z_N}\,
e^{- N h_a} \, \bs \mu (a) \;. 
\end{equation}

The next result and Theorem \ref{vs04} below are discrete versions of
a result of Bovier, Eckhoff, Gayrard and Klein \cite{begk3}. The proofs
are based on the proof of Theorem 3.1 in \cite{begk3} and on
\cite{bbi1, bbi2}.

\begin{theorem}
\label{vs05}
For every proper subset $A$ of $S$,
\begin{equation*}
\lim_{N\to\infty}
\frac{Z_N }{(2\pi N)^{d/2}}  \, 2\pi N\, e^{ N  H_i}\,  
\Cap_N (\ms E_N(A), \ms E_N(A^c )) \;=\;
\sum_{\bs z\in \mf S(A)} 
\,  \frac{\mu (\bs z)}{\sqrt{- \det {\rm Hess}\, F(\bs z)}} \;\cdot
\end{equation*}
\end{theorem}

This result together with two other estimates permit to prove the
metastable behavior of the Markov chain $X_N(t)$ among the shallowest
valleys $\ms E^a_N$. To examine the metastable behavior of the chain
$X_N(t)$ on deeper wells we need to extend Theorem \ref{vs05} to
disjoint sets $A$, $B$ which do not form a partition of $S$, $A\cup B
\not = S$. The statement of this extension and its proof requires the
introduction of a graph.

\smallskip\noindent{\bf B. A Graph associated to the chain.}  
Let $\bb G = (S, E)$ be the weighted graph whose vertices are $S=\{1,
\dots, \ell\}$, the indices of the sets $W_{a}$. Place an edge between
$a$ and $b\in S$ if and only if there exists a saddle point $\bs z$
belonging to $W_{a} \cap W_{b}$, i.e., if $\mf S_{a,b} \not =
\varnothing$. The weight of the edge between $a$ and $b$, denoted by
$\bs c(a,b)$, is set to be
\begin{equation}
\label{v22}
\bs c(a,b) \;=\; \sum_{\bs z\in \mf S_{a,b}} 
\,  \frac{\mu (\bs z)}{\sqrt{- \det {\rm Hess}\, F(\bs z)}} \;.    
\end{equation}
Note that $\bs c(a,b)$ vanishes if there is no saddle point $\bs z$
belonging to $W_{a} \cap W_{b}$ and that the weights are independent
of $N$. Figure \ref{fig2} present the weighted graph associated to
one of the connected component of Figure \ref{fig1}. 

\begin{figure}[htb]
  \centering
  \def\svgwidth{340pt}
  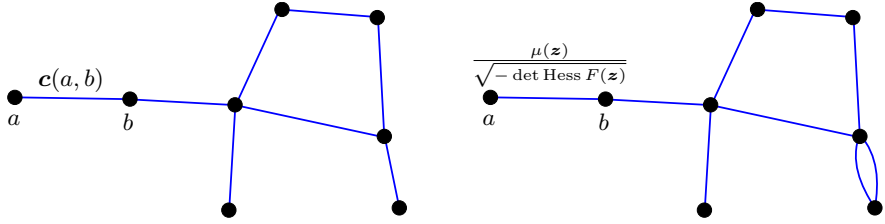
  \caption{The simple weighted graph and the graph with multiple edges
associated to one of the connected components of Figure~\ref{fig1}.}
\label{fig2}
\end{figure}

The graph $\bb G$ has to be interpreted as an electrical network,
where the weights $\bs c(a,b)$ represent the conductances.  It would
be more natural to start with a graph with multiple edges, each edge
corresponding to a saddle point $\bs z$. However, adding the parallel
conductances one can reduce the graph with multiple edges to the above
graph.

Let
\begin{equation*}
c_N(a,b) \;=\; \frac{(2\pi N)^{d/2}}{Z_N }\,
\frac{e^{- N  H_i}}{2 \pi N} \, \bs c(a,b)  \;,\;\; a\;,\; b\in S\;.
\end{equation*}
It follows from Theorem \ref{vs05} and from a calculation
that
\begin{equation}
\label{v21}
\begin{split}
& c_N (a,b) \;=\; \\
&\quad [1+o_N(1)] \, \frac 12 \Big\{ \Cap_N(\mc E^{a}_N, \breve{\mc
E}^{a}_N) \;+\; \Cap_N(\mc E^{b}_N, \breve{\mc E}^{b}_N) 
\;-\; \Cap_N \big(\mc E^{a}_N \cup \mc E^{b}_N, \cup_{c\not = a,b}
\mc E^{c}_N \big) \Big\}\;,
\end{split}
\end{equation}
where, $\breve{\mc E}^{a}_N \;=\; \cup_{c\not = a} \mc E^{c}_N$.  This
explains the definition of $c_N(a,b)$. Moreover, by \cite[Lemma
6.8]{bl2}, $c_N(a,b)$ is equal to $\mu_N(\mc E^{a}_N) r_N (\mc
E^{a}_N, \mc E^{b}_N)$, where $r_N = r^1_N$ represents the average
rates introduced below in \eqref{v28}.

For two disjoint subsets $A$, $B$ of $S$, denote by $\Cap_{\bb G}
(A,B)$ the \emph{conductance} between $A$ and $B$. To define the
conductance, denote by $\{Y_k : k\ge 0\}$ the discrete-time random
walk on $S$ which jumps from $a$ to $b$ with probability
\begin{equation}
\label{v42}
p(a,b) \;=\; \frac{\bs c(a,b)}{\sum_{b'\in S} \bs c (a,b')}\;\cdot
\end{equation}
Denote by $\bb P^Y_a$, $a\in S$, the distribution of the chain $Y_k$
starting from $a$ and by $V_{A,B}$, $A$, $B\subset S$, $A\cap
B=\varnothing$, the equilibrium potential between $A$ and $B$:
\begin{equation*}
V_{A,B} (b) \;=\; \bb P^Y_b \big[ H_A < H_B \big]\;,\;\; b\in S\;,
\end{equation*}
where $H_C$, $C\subset S$, represents the hitting time of $C$: $H_C =
\min\{k\ge 0 : Y_k\in C\}$. The conductance between $A$ and $B$ is
defined as
\begin{equation*}
\Cap_{\bb G} (A,B) \;=\; \frac 12 \sum_{a,b\in S} \bs c (a,b)
\big[V_{A,B} (b)  - V_{A,B} (a) \big]^2 \;.
\end{equation*}
By \cite[Proposition 3.1.2]{g1} the conductance between $A$ and $B$
coincides with the capacity between $A$ and $B$. The next result
establishes that the capacities for the chain $X_N(t)$ can be computed
from the conductances on the finite graph $\bb G$.

\begin{theorem}
\label{vs04}
For every disjoint subsets $A$, $B$ of $S$,
\begin{equation*}
\Cap_N (\ms E_N(A), \ms E_N(B) ) \;=\; [1+o_N(1)]\,  
\frac{(2\pi N)^{d/2}}{Z_N }\,
\frac{e^{- N  H_i}}{2 \pi N} \, \Cap_{\bb G} (A,B)\;.
\end{equation*}
\end{theorem}

\begin{remark}
\label{vs21}
It follows from the proofs of Theorems \ref{vs05} and \ref{vs04} that
both statements remain in force if we replace the sets $\ms E^a_N$ by
singletons $\{\bs x^a_N\}$, where $\bs x^a_N\in \ms E^a_N$. In this case the
sets $\ms E_N(A)$ become $\{\bs x^a_N : a\in A\}$.
\end{remark}

\smallskip\noindent{\bf C. Metastability.}  The Markov chain $X_N(t)$
exhibits a metastable behavior among the wells of each connected
component $\varOmega^i_j$.  The description of this behavior requires
some further notation.

Recall that $h_a = F(\bs m_{a,1})$ represents the value of $F$ at a
deepest minima of the well $W_a$.  Let $\hat \theta_{a} = H_i - h_a
>0$, $a\in S$, be the depth of the well $W_{a} $.  The depths $\hat
\theta_{a}$ provide the time-scale at which a metastable behavior is
observed. Let $\theta_{1} < \theta_{2} < \cdots < \theta_{n}$, $n\le
\ell$, be the increasing enumeration of the sequence $\hat
\theta_{a}$, $1\le a\le \ell$:
\begin{equation*}
\{\hat \theta_{1}, \dots, \hat \theta_{\ell}\} \;=\; 
\{\theta_{1}, \dots, \theta_{n}\}\;.
\end{equation*}
Of course, $n$ and $\theta_m$ depend on the component $\varOmega^i_j$.
If we need to stress this dependence, we will denote $n$, $\theta_m$
by $n_{i,j}$, $\theta^{i,j}_m$, respectively.

The chain exhibits a metastable behavior on $n$ different time scales
in the set $\varOmega$. Let $T_m = \{a\in S : \hat \theta_{a} =
\theta_m\}$, $1\le m\le n$, so that $T_1, \dots, T_n$ forms a
partition of $S$, and let
\begin{equation*}
S_m \;=\; T_m \cup \cdots \cup T_n\;,\;\; 1\le m\le n\;.
\end{equation*}
Define the projection $\Psi^m_N : \Xi_N\to S_m \cup \{ N\}$, $1\le
m\le n$, as
\begin{equation}
\label{v25}
\Psi^m_N (\bs x) \;=\; \sum_{a\in S_m} a \, \mb 1 \big\{\bs x \in \ms
E^{a}_N \big\} \;+\; N\, \mb 1\Big\{\bs x \not\in \bigcup_{a\in S_m} \ms
E^{a}_N \Big\} \;.
\end{equation}
Denote by $\bs X^m_N(t)$ the projection of the Markov chain $X_N(t)$
by $\Psi^m_N$:
\begin{equation*}
\bs X^m_N(t) \;=\; \Psi^m_N (X_N(t))\;.
\end{equation*}

Fix $1\le m\le n$.  We introduce some notation to define the
asymptotic dynamics of the process $\bs X^m_N(t)$. The time scale in
which the process $\bs X^m_N$ evolves, denoted by $\beta_m = \beta_m
(N)$, is given by
\begin{equation*}
\beta_m \;=\; 2\pi N \, e^{\theta_m N}\;.
\end{equation*}

For $a$, $b$ in $S_m$, let 
\begin{equation}
\label{v41}
\bs c_m (a,b) \;=\; \frac 12 \Big\{ \Cap_{\bb G}( \{a\}, S_m\setminus
\{a\}) \;+\; \Cap_{\bb G}( \{b\}, S_m\setminus \{b\}) -
\Cap_{\bb G}( \{a,b\}, S_m\setminus \{a, b\}) \Big\} \;.
\end{equation}
Note that $\bs c_m(a,b)$ represents the conductance between $a$ and $b$
for the electrical circuit obtained from $\bb G$ by removing the
vertices in $S^c_m$. In particular, $\bs c_1 (a,b) = \bs c (a,b)$ for
$a$, $b\in S_m$. Let
\begin{equation}
\label{v23}
\bs r_m (a,b) \;=\;
\begin{cases}
\bs c_m (a,b)/\bs \mu (a) & a\in T_m \,,\, b\in S_m \;, \\
0 & a\in S_{m+1} \,,\, b\in S_m\;.
\end{cases}
\end{equation}
Recall from \cite{l2} the definition of the soft topology.

\begin{theorem}
\label{vs07}
Fix $1\le i \le i_0$, $1\le j\le \ell_i$, $1\le m\le n_{i,j}$, $a\in
S_m$ and a sequence of configurations $\bs x_N$ in $\ms
E^{a}_N$. Under $\mb P_{\bs x_N}$, the time re-scaled projection $\bb
X^m_N(t) = \bs X^m_N(t \beta_m)$ converges in the soft topology to a
$S_m$-valued continuous-time Markov chain $\bb X^m(t)$ whose jump
rates are given by \eqref{v23}. In particular, the points in $S_{m+1}$
are absorbing for the chain $\bb X^m(t)$.
\end{theorem}

\begin{remark}
\label{vs08}
Theorem \ref{vs07} states that the weighted graph $\bb G$, the measure
$\bs \mu$ and the sequence $\beta_m(N)$ describe the evolution of the
chain $X_N(t)$ in the connected component $\varOmega$. The weighted graph
with multiple edges would describe more accurately the chain $X_N(t)$,
providing the probability that the chain leaves a well $W_a$ through a
mesoscopic neighborhood of a saddle point $\bs z \in \mf S$. This
statement is made precise in Theorem \ref{vs20} below.
\end{remark}

\begin{remark}
\label{vs06}
Nothing prevent two time-scales at different levels to be equal, or
two time scales in different connected components of the same level to
be equal. It is possible that $\theta^{i,j}_{m} = \theta^{i',j'}_{m'}$
for some $i\not = i'$ or that $\theta^{i,j}_{m} = \theta^{i,j'}_{m'}$
for some $j\not = j'$. 
\end{remark}

\smallskip\noindent{\bf D. Exit points from a well.}
Fix $1\le i \le i_0$, $1\le j\le \ell_i$, and recall that we denote by
$W_a = W^i_{j,a}$, $a\in S = \{1, \dots, \ell^i_j\}$, the wells which
form the connected component $\varOmega^i_j$. The last result of this
article states that the chain $X_N(t)$ leaves the set $W_a$ through a
neighborhood of a saddle point $\bs z$ in the boundary of $W_a$ with
probability $\omega (\bs z)/\sum_{\bs z'} \omega (\bs z')$, where the
summation is carried over all saddle points in the boundary of $W_a$
and where
\begin{equation}
\label{v44}
\omega (\bs z) \;=\;
\frac{\mu (\bs z)}{\sqrt{- \det {\rm Hess}\, F(\bs z)}} \;\cdot 
\end{equation}

Let $\delta_N$ be a sequence such that $\delta_N \ll N^{-3/4}$,
$N^{d+1} \exp\{- N \delta_N\}\to 0$. Denote by $\varOmega_N =
\varOmega^i_{j, N}$ the connected component of the set $\{\bs x\in \Xi
: F(\bs x) \le H_i + \delta_N\}$ which contains $\varOmega^i_j$. Since
$\delta_N\downarrow 0$, for $N$ large enough, $\varOmega^i_{j, N} \cap
\varOmega^i_{j',N} = \varnothing$ for all $j'\not = j$. In particular,
for $N$ large enough there is a one-to-one correspondance between
$\varOmega^i_{j}$ and $\varOmega^i_{j, N}$.

Fix $a\in S$ and let $\mf S_a$ be the set of saddle points in the
boundary of $W_a$, $\mf S_a = \cup_{b\in S, b\not = a} \mf S_{a,b}$.
Denote by $\partial \varOmega_N$ the boundary of $\varOmega_N$ and by
$B_\epsilon(\bs x)$ the open ball of radius $\epsilon>0$ around $\bs
x\in \Xi$.  We modify the set $\partial \varOmega_N$ around each
saddle point $\bs z\in \mf S_a$ to obtain a closed manifold
$D_a\subset \varOmega_N$.

Fix a saddle point $\bs z\in \mf S_a$ and recall condition (H3) on
$F$. Denote by $-\mu < 0 < \lambda_2 \le \cdots \le \lambda_d$ the
eigenvalues of ${\rm Hess}\, F(\bs z)$, and by $\bs v$, $\bs w^i$,
$2\le i\le d$, an associated orthonormal basis of eigenvectors. Let
$\bb H = \bb H_{\bs z}$ be the $(d-1)$-dimensional hyperplane
generated by the vectors $\bs w^i$, $2\le i\le d$. By a Taylor
expansion, there exists $\epsilon>0$ such that 
\begin{equation}
\label{v45}
F(\bs x) \;\ge\; H_i \;+\; 
\frac { \lambda_2}4 \, \Vert \bs x- \bs z\Vert^2
\end{equation}
for $\bs x \in \bs z + \bb H =\{\bs z + \bs y : \bs y\in \bb H\}$ such
that $\Vert \bs x - \bs z\Vert \le \epsilon$. Let
\begin{equation}
\label{v52}
D_{\bs z} \;=\; \big\{\bs y \in (\bs z + \bb H) \cap B_\epsilon (\bs z) : 
F(\bs y) \le H_i + \delta_N \big\}\;.
\end{equation}
We intersected the set $\bs z + \bb H$ with the set $B_\epsilon (\bs
z)$ to avoid including in $D_{\bs z}$ points which are far from $\bs
z$. 

The set $D_a = D_a^N$ is defined as follows. For each $\bs z\in\mf
S_a$, remove from $\partial \Omega_N$ the set $(\bs z + \bb H)
\cap \partial \Omega_N \cap B_\epsilon (\bs z)$. As before, the set
$B_\epsilon (\bs z)$ has been introduced to avoid removing from
$\partial \Omega_N$ points which are far from $\bs z$. Denote by
$\Omega^1_N$ the set obtained after this operation, which is a finite
union of connected sets. Remove from $\Omega^1_N$ all connected
component which contain a point close to some saddle point which does
not belong to $\mf S_a$.  Denote this new set by $\Omega^2_N$. $D_a$
is the union of $\Omega^2_N$ with all set $D_{\bs z}$, $\bs z \in\mf
S_a$:
\begin{equation*}
D_a \;=\; \bigcup_{\bs z \in\mf S_a} D_{\bs z} \;\cup\; \Omega^2_N \;.
\end{equation*}

Let $\ms D_a$, $\ms D_{\bs z}\subset \Xi_N$, be discretizations of the
sets $D_a$, $D_{\bs z}$: $\ms D_a = \{\bs x\in \Xi_N : d(\bs x, D_a)
\le 1/N\}$, where $d$ stands for the Euclidean distance, $d(\bs x, A)
= \inf_{\bs y\in A} \Vert \bs x - \bs y\Vert$.

\begin{theorem}
\label{vs20}
Fix $1\le i \le i_0$, $1\le j\le \ell_i$, and $a\in S = \{1, \dots
\ell^i_j\}$.  Let .  For all $\bs z\in \mf S_a$, and all sequences 
$\{\bs x_N : N\ge 1\}$, $\bs x_N\in \ms E^a_N$,
\begin{equation*}
\lim_{N\to\infty} \mb P_{\bs x_N} \Big[ H_{\ms D_a} = 
H_{\ms D_{\bs z}} \Big]  \;=\; \frac{\omega (\bs z)}
{\sum_{\bs z'\in \mf S_a} \omega (\bs z')}\;\cdot
\end{equation*}
\end{theorem}

The proof of Lemma \ref{vs22} yields the last result.

\begin{proposition}
\label{vs27}
Let $D\subset\Xi$ be a domain with a smooth boundary, and let $m
=:\inf_{\bs y\in\partial D} F(\bs y)$. Fix a sequence $\{\epsilon_N :
N\ge 1\}$ of positive numbers such that $\lim_N N^{d+1}$ $\exp\{-N
\epsilon_N\} =0$, and let $D_N = \Xi_N \cap D$, $B_N=\{\bs
x\in \partial D_N : F(\bs x) \le m + 2 \epsilon_N\}$.  Fix a point
$\bs x\in D$ such that $F(\bs x) < m$ and for which there exists a
continuous path $\bs x(t)$, $0\le t\le 1$, from $B_N$ to $\bs x$ such
that $F(\bs x(t)) \le m + \epsilon_N$ for all $0\le t\le 1$. Then,
\begin{equation*}
\lim_{N\to\infty} \mb P_{\bs x_N} \big[ H_{\partial D_N} = H_{B_N}
\big]\;=\; 1\;,
\end{equation*}
where $\bs x_N\in D_N$, $\Vert \bs x_N - \bs x\Vert\le 1/N$.
\end{proposition}

We conclude this section with some comments.  Bianchi, Bovier and
Ioffe \cite{bbi1, bbi2} examined the metastable behavior of the
Curie-Weiss model with random external fields. In this case the
potential $F$ becomes a sequence of potentials $F_N$ which converges
to some function $F_\infty$. The authors assumed that the parameter of
the model, the distribution of the external field, were chosen to
guarantee that all wells do not have saddle points at the same
height. In this case, the metastable behavior of the chain consists in
staying for an exponential time in some well and then to jump to a
deeper well in which the chain remains trapped for ever.

To observe a metastable behavior similar to the one described in
Theorem \ref{vs07}, one has to tune the distribution of the external
field in a way that the wells associated to $F_\infty$ have more than
one saddle point at the same height.  In this case, however, the
metastable behavior might depend on the subsequence of $N$.

To illustrate this possibility, consider the following one-dimensional
example. Let $F_N$ be a sequence of potentials which converge
uniformly to a potential $F_\infty$. Fix two local maxima of
$F_\infty$, supposed to be at the same height, $F_\infty(\bs z) =
F_\infty(\bs z')$, and assume that the interval $(\bs z, \bs z')$ is a
well, $F_\infty (\bs x) < F_\infty(\bs z)$ for $\bs z<\bs x< \bs z'$.
Suppose also that $F_N$ has two local maxima $\bs z_N$, $\bs z'_N$
such that $\bs z_N\to \bs z$, $\bs z'_N\to \bs z'$, that $(\bs z_N,
\bs z'_N)$ is a well for $F_N$, and that there exists subsequences
$N'$ and $N''$ such that
\begin{equation*}
N' \big[ F_{N'}(\bs z'_{N'}) - F_{N'}(\bs z_{N'})\big] \;\le\; - \epsilon
\;, \quad
N'' \big[ F_{N''}(\bs z'_{N''}) - F_{N''}(\bs z_{N''})\big] \;\ge\; \epsilon
\end{equation*}
for some $\epsilon>0$. In this case, in view of the results presented
in this section, starting from a local minima in $(\bs z_N, \bs
z'_N)$, along the subsequence $N'$, almost surely the chain will
escape from $(\bs z_N, \bs z'_N)$ through a neighborhood of $\bs z'_N$,
while along the subsequence $N''$ almost surely it will escape from
$(\bs z_N, \bs z'_N)$ through a neighborhood of $\bs z_N$.

This is what happens for the Curie-Weiss model with an external field,
random or not, if there exist saddle points at the same height. For
the metastable behavior not to depend on particular subsequences, one
needs to impose some strong conditions on the asymptotic behavior of
the sequence $F_N$.

The article is divided as follows. In Section \ref{vsec02} we prove
the upper bound for the capacities appearing in the statement of
Theorem \ref{vs05} and in Section \ref{vsec01} the lower bound. In
Section \ref{vsec03} we prove Theorem \ref{vs04}, in Section
\ref{vsec05}, Theorem \ref{vs07}, and in Section \ref{vsec06}, Theorem
\ref{vs20}. 

\section{Upper bound for the capacities}
\label{vsec02}

We prove in this section the upper bound of Theorem \ref{vs05}.  The
proof is based on ideas of \cite{begk3, bbi1, bbi2} and on the
Dirichlet principle \cite[Proposition 3.1.3]{g1} which expresses the
capacity between two sets as an infimum of the Dirichlet form: for two
disjoint subsets $A$, $B$ of $\Xi_N$,
\begin{equation*}
\Cap_N (A,B) \;=\; \inf_f D_N(f)\;,
\end{equation*}
where the infimum is carried over all functions $f:\Xi_N \to \bb R$
such that $f(\bs x) = 1$, $\bs x\in A$, $f(\bs y)=0$, $\bs y\in B$,
and where $D_N(f)$ stands for the Dirichlet form of $f$,
\begin{equation*}
D_N(f) \;=\; \sum_{\bs x\in \Xi_N} f(\bs x) \, (-L_N f)(\bs x) \,
\mu_N(\bs x) \,=\, \frac 12 \sum_{\bs x, \bs y \in \Xi_N} 
\mu_N(\bs x) R_N(\bs x, \bs y) [f(\bs y) -f(\bs x)]^2 \;.
\end{equation*}

\begin{proposition}
\label{vs01} 
For every proper subset $A$ of $\{1, \dots, \ell\}$,
\begin{equation*}
\limsup_{N\to\infty}
\frac{Z_N }{(2\pi N)^{d/2}}  \, 2\pi N\, e^{ N  F (\bs z)}\,  
\Cap_N (\ms E_N(A), \ms E_N(A^c )) \;\le\;
\sum_{\bs z\in \mf S(A)} 
\,  \frac{\mu (\bs z)}{\sqrt{- \det {\rm Hess}\, F(\bs z)}} \;\cdot
\end{equation*}
\end{proposition}

The proof of this proposition is divided in several lemmas. The main
point is that the capacities depend on the behavior of the function
$F$ around the saddle points of $F$.

Fix a saddle point $\bs z$ of $F$ and denote by $\bb M = ({\rm Hess}\,
F)(\bs z)$ the Hessian of $F$ at $\bs z$. Denote by $-\mu$ the
negative eigenvalue of $\bb M$ and by $0<\lambda_2 \le \cdots \le
\lambda_{d}$ the positive eigenvalues. Let $\bs v$, $\bs w^i$, $2\le
i\le d$, be orthonormal eigenvectors associated to the eigenvalues
$-\mu$, $\lambda_i$, respectively. We sometimes denote $\bs v$ by $\bs
w^1$ and $-\mu$ by $\lambda_1$.

Let $\bb V$ the $(d\times d)$-matrix whose $j$-th column is the vector
$\bs w^j$ and denote by $\bb V^*$ its transposition. Denote by
$\bb D$ the diagonal matrix whose diagonal entries are $\lambda_i$
so that $\bb M = \bb V \bb D \bb V^*$. Let $\bb D_\star$ be the matrix
$\bb D$ in which we replaced the negative eigenvalue $\lambda_1$ by
its absolute value $\mu$ and let
\begin{equation}
\label{v04}
\bb M_\star \;=\; \bb V \,\bb D_\star \,\bb V^*\;.
\end{equation}
Clearly, $\det \bb M = - \det \bb M_\star$.

Let $B_N = B^{\bs z}_N = \ms B^{\bs z}\cap \Xi_N$ be a mesoscopic
neighborhood of $\bs z$:
\begin{equation}
\label{v05}
\ms B^{\bs z} = \Big \{\bs x \in \Xi : 
\vert (\bs x - \bs z) \cdot \bs v \vert  \le \varepsilon_N 
\,,\, \max_{2\le j\le d} 
\vert (\bs x - \bs z) \cdot \bs w^j \vert  \le 
2 \,\sqrt{\mu/\lambda_j} \varepsilon_N \, \Big\}\;,
\end{equation}
where $N^{-1} \ll \varepsilon_N\ll 1$ is a sequence of positive
numbers to be chosen later. Unless needed, we omit the index $\bs z$
from the notation $B^{\bs z}_N$. Denote by $\partial B_N$ the outer
boundary of $B_N$ defined by
\begin{equation}
\label{v47}
\partial B_N = \{\bs x \in \Xi_N \setminus B_N : \exists\,
\bs y \in B_N  \text{ s.t. } \Vert\bs y - \bs x\Vert=N^{-1} \}\;,
\end{equation}
and let $\partial_- B_N$, $\partial_+ B_N$ be the pieces of the outer
boundary of $B_N$ defined by
\begin{equation*}
\begin{split}
& \partial_- B_N = \{\bs x \in \partial B_N :  
(\bs x - \bs z) \cdot \bs v < - \varepsilon_N \}\;, \\
&\qquad \partial_+ B_N = \{\bs x \in \partial B_N :  
(\bs x - \bs z) \cdot \bs v >  \varepsilon_N \}\;.
\end{split}
\end{equation*}

\subsection*{The Dirichlet forms in the sets $B_N$}
Denote by $D_{N}(f;B_N)$ the piece of the Dirichlet form of a function
$f:\Xi_N\to\bb R$ corresponding to the edges in the set $B_N$:
\begin{equation*}
D_{N} (f ; B_N)\;=\; 
\sum_{i=1}^d \sum_{\bs x\in B_N} \mu_N (\bs x)\, R_N (\bs x, \bs x +
\bs e_i)\, [f(\bs x + \bs e_i) - f (\bs x)]^2\;,
\end{equation*}
where $\{\mf e_1, \dots, \mf e_d\}$ is the canonical basis of $\bb
R^d$ and $\bs e_i = N^{-1} \mf e_i$.

Let $\mu^{\bs z}_N$ be the measure on $B_N$ given by
\begin{equation*}
\mu^{\bs z}_N  (\bs x) \;=\; \frac 1{Z_N}\,
e^{- N F(\bs z)}
\, e^{ - (1/2) N (\bs y \cdot \, \bb M \bs y)} \;,
\end{equation*}
where $\bs y = \bs x - \bs z$, and where $\bs v \cdot \bs w$
represents the scalar product between $\bs v$ and $\bs w$.  Denote by
$D^{\bs z}_N$ the Dirichlet form defined by
\begin{equation}
\label{v02}
D^{\bs z}_N  (f)\;=\; 
\sum_{i=1}^d \sum_{\bs x\in B_N} \mu^{\bs z}_N (\bs x)\, \,
[f(\bs x + \bs e_i) - f (\bs x)]^2\;.
\end{equation}
The next assertion follows from an elementary computation and from
assumption (H1).

\begin{asser}
\label{vs14}
For every function $f:\Xi_N\to \bb R$,
\begin{equation*}
D_{N} (f ; B_N)\;=\; \big[1+ O(\varepsilon_N) 
+ O(N\varepsilon_N^3) \big] \,  D^{\bs z}_N  (f)\;.
\end{equation*}
\end{asser}

\subsection*{The equilibrium potential} We introduce in this
subsection an approximation in the set $B_N$ of the solution to the
Dirichlet variational problem for the capacity. To explain the choice,
consider a one-dimensional random walk on the interval $I_N =
\{-K_N/N, \dots$, $(K_N-1)/N, K_N/N\}$ whose Dirichlet form $D_N$ is
given by
\begin{equation*}
D_N(f) \;=\; \sum_k e^{\mu N k^2} \, [f(k+N^{-1}) -
f(k)]^2\;,
\end{equation*}
where the sum is performed over $k\in I_N$, $k\not = K_N/N$.  An
elementary computation shows that the equilibrium potential $V(k/N) =
P_{k/N}[H_{K_N/N} < H_{-K_N/N}]$ is given by
\begin{equation*}
V(k/N) \;=\; \frac{\sum_{j=-K_N/N}^{(k-1)/N} e^{- \mu N j^2}}
{\sum_{j=-K_N/N}^{(K_N-1)/N} e^{- \mu N j^2}}\;\sim\;
\frac{\int_{-\infty}^{k/N}  e^{- \mu N r^2} \, dr}
{\int_{-\infty}^{\infty} e^{- \mu N r^2} \, dr}\;,
\end{equation*}
where the last approximation holds provided $\sqrt{N} \ll K_N \ll N$.

In view of the previous observation, let $f:\bb R\to\bb R_+$ be given
by
\begin{equation*}
f_N(r) \;=\; \frac{\int_{-\infty}^r e^{- (1/2) N \mu s^2} \, ds}
{\int_{-\infty}^\infty e^{- (1/2) N \mu s^2} \, ds}
\;=\; \sqrt{\frac{N\mu}{2\pi}} 
\int_{-\infty}^r e^{- (1/2) N \mu s^2} \, ds \;.
\end{equation*}
The function $V_N$ defined below is an approximation on the set $B_N$
for the equilibrium potential between $\partial_- B_N$ and $\partial_+
B_N$:
\begin{equation}
\label{v03}
V_N(\bs x) \;=\; V^{\bs z}_N(\bs x) \;=\;  f_N([\bs x - \bs z] \cdot \bs v)\;.
\end{equation}

\begin{asser}
\label{vs02}
Assume that $N^{-1/2}\ll\varepsilon_N \ll N^{-1/3}$. Then,
\begin{equation*}
\frac{Z_N }{(2\pi N)^{d/2}}   \, 2 \pi N \, e^{ N  F (\bs z)}\,  
D_{N}(V_N; B_N) \;=\; [1 + o_N(1) ]
\,  \frac{\mu}{\sqrt{- \det {\rm Hess}\, F(\bs z)}} \;\cdot
\end{equation*}
\end{asser}

\begin{proof}
By Assertion \ref{vs14}, it is enough to estimate $D^{\bs z}_N
(V_N)$. By definition of the Dirichlet form $D^{\bs z}_N$,
\begin{equation*}
Z_N \, e^{ N  F (\bs z)}\,  D^{\bs z}_N (V_N) \;=\; 
\sum_{i=1}^d \sum_{\bs x\in B_N} 
e^{ - (1/2) \, N\, (\bs y \cdot \bb M \, \bs y)}\,
[V_N(\bs x + \bs e_i) - V_N(\bs x)]^2\;,
\end{equation*}
where $\bs y = \bs x-\bs z$. Denote by $\bs v_1, \dots, \bs v_d$ the
coordinates of the vector $\bs v$ and recall that $\Vert \bs
v\Vert=1$.  Recall the definition of the matrix $\bb M_\star$
introduced in \eqref{v04}. Since $(\bs y \cdot \bb M \, \bs y) =
\sum_{1\le j\le d} \lambda_j (\bs y \cdot \bs w^j)^2$, by definition
of $V_N$ this sum is equal to
\begin{equation*}
\begin{split}
& [1 + o_N(1)] \, \frac{\mu}{2 \pi N}\, 
\sum_{i=1}^d \bs v^2_i \sum_{\bs x\in B_N} 
e^{ - (1/2) \, N \, (\bs y \cdot \bb M \, \bs y)}\,
e^{- \mu N (\bs y\cdot \bs v)^2} \\
&\quad =\; 
[1 + o_N(1)] \, \frac{\mu}{2 \pi N}
\sum_{\bs x\in B_N} 
e^{ - (1/2) \, N \, (\bs y \cdot \bb M_\star \, \bs y) }\;.
\end{split}
\end{equation*} 
Let $\bs w = \sqrt{N} \bs y = \sqrt{N}[\bs x - \bs z]$ so that $\bs w
\in N^{-1/2} \bb Z^d$, to rewrite the previous sum as
\begin{equation*}
[1 + o_N(1)] \, \frac{\mu}{2 \pi N}
\sum_{\bs w}  e^{ - (1/2) \, (\bs w \cdot \bb M_\star \, \bs w) } \;,
\end{equation*}
where the sum is performed over $\bs w$ such that $|\bs w\cdot \bs
v|\leq N^{1/2}\varepsilon_N$ and $| \bs w \cdot \bs w^j | \leq 2
\,\sqrt{\mu/\lambda_j} N^{1/2}\varepsilon_N$, $2\leq j \leq d$. Since,
by assumption, $N^{1/2}\varepsilon_N\uparrow\infty$, this expression
is equal to
\begin{equation*}
[1 + o_N(1)] \, \frac{\mu}{2 \pi N} 
N^{d/2} \int_{\bb R^d}
e^{ - (1/2) \, (\bs w \cdot \bb M_\star \bs w)} d\bs w \;.  
\end{equation*}
The previous integral is equal to $(2\pi)^{d/2} \{ \det \bb
M_\star\}^{-1/2} = (2\pi)^{d/2} \{ - \det \bb M\}^{-1/2}$, which
completes the proof of the assertion.
\end{proof}

We conclude the proof of Proposition \ref{vs01} extending the
definition of $V_N$ to the entire set $\Xi_N$ and estimating its
Dirichlet form.  We denote by $\partial^{\rm in} B_N$ the inner
boundary of $B_N$, the set of points in $B_N$ which have a neighbor in
$\Xi_N\setminus B_N$. Let $B^*_N$, $\partial^{\rm in}_\pm B_N$ be the
$(d -1)$-dimensional sections of the boundary $\partial B_N$,
$\partial^{\rm in} B_N$:
\begin{equation*}
\begin{split}
& B^*_N \;=\; \bigcup_{2\le j\le d} \big\{\bs x \in \partial B_N : | 
(\bs x - \bs z) \cdot \bs w^j | > 2
\sqrt{\mu/\lambda_j}\varepsilon_N \big\}\;, \\
& \qquad \partial^{\rm in}_\pm B_N = \{\bs x \in \partial^{\rm in} B_N :  
\exists \, \bs y\in \partial_\pm B_N \text{ s.t. } \Vert \bs y -\bs x
\Vert = N^{-1} \}\;.
\end{split}
\end{equation*}

\begin{asser}
\label{vs15}
For all $N$ sufficiently large,
\begin{equation*}
\inf_{\bs x\in B^*_N} F(\bs x) \;\ge\; F(\bs z) \;+\; \mu\, \varepsilon^2_N\;. 
\end{equation*}
\end{asser}
 
\begin{proof}
Indeed, by a Taylor expansion of $F$ around $\bs z$, for $\bs x
\in B^*_N$,
\begin{equation*}
F(\bs x) \;=\; F(\bs z) \;+\; (1/2) (\bs x-\bs z) \cdot \bb M
\, (\bs x- \bs z) \;+\; O(\varepsilon^3_N)\;.
\end{equation*}
The second term on the right hand side is equal to
$(1/2) \sum_{1\le j\le d} \lambda_j \, [(\bs x-\bs z)\cdot
\bs w^j]^2$.  Since $\lambda_1 = - \mu$, $\lambda_j>0$ for
$2\le j\le d$, and $\bs x$ belongs to $B^*_N$, for $N$ sufficiently
large the previous expression is bounded below by
\begin{equation*}
F(\bs z) \;+\; \frac {3\, \mu}2 \, \varepsilon^2_N 
\;+\; O(\varepsilon^3_N) \;\ge \;
F(\bs z) \;+\; \mu \, \varepsilon^2_N \;,
\end{equation*}
which proves the claim.
\end{proof}

Let $\vartheta=\min \{\mu(\bs z) : \bs z \in \mf S (A)\}$. Denote by
$\ms U$ the connected component of the set $\{\bs x\in \Xi : F(\bs x)<
F(\bs z) + \vartheta \, \varepsilon^2_N\}$ which contains a set $W_a$,
$a\in A$. The set $\ms U$ may be decomposed in disjoint sets. Recall
from \eqref{v05} the definition of the sets $\ms B^{\bs z}$, $\bs z
\in \mf S (A)$, and let $\ms V = \ms U \setminus \cup_{\bs z \in \mf S
  (A)} \ms B^{\bs z}$. Figure \ref{fig3} represents the sets $\ms U$
and $\ms B^{\bs z}$. By Assertion \ref{vs15}, the set $\ms V$ is
formed by several connected components separated by the sets $\ms B^{\bs
  z}$, $\bs z \in \mf S (A)$. In Figure~\ref{fig3}, for example, the
set $\ms V$ is composed of $4$ connected components.

Let $\ms U_N = \ms U \cap \Xi_N$, $\mf B^{\bs z}_N = \ms U_N\cap \ms
B^{\bs z}$, $\bs z \in \mf S (A)$, $\ms V_N = \ms U_N \setminus
\cup_{\bs z \in \mf S (A)} \mf B^{\bs z}_N$ so that
\begin{equation*}
\ms U_N \;=\; \ms V_N \,\cup\, \bigcup_{\bs z \in \mf S (A)} 
\mf B^{\bs z}_N\;.
\end{equation*}
Let $\ms V^A_N$ be the union of all connected components of $\ms V_N$
which contains a point in $W_a$, $a\in A$, and let $\ms V^B_N = \ms
V_N \setminus \ms V^A_N$.

\begin{figure}[htb]
  \centering
  \def\svgwidth{330pt}
  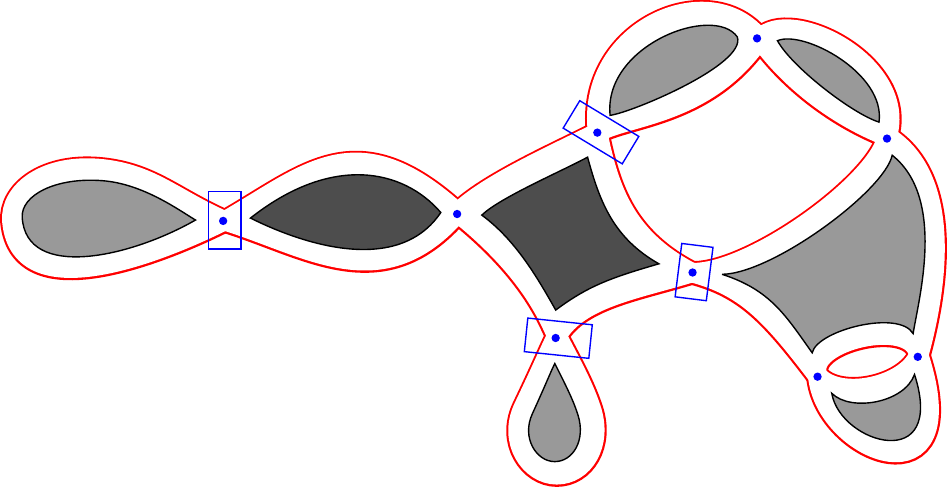
  \caption{In red the boundary of the set $\ms U$.  In dark gray the
    wells $W^\epsilon_b$, $b\in A$. In blue the boxes $\ms B^{\bs z}$,
    $\bs z \in \mf S(A)$.}
\label{fig3}
\end{figure}

For each $\bs z \in \mf S (A)$, choose an orthonormal basis of $({\rm
  Hess}\, F)(\bs z)$ in such a way that the eigenvector $\bs v (\bs
z)$ points to the direction of $\ms E_N(A)$.  Define $V^A_N: \Xi_N \to
[0,1]$ by
\begin{equation*}
V^A_N(\bs x) \;=\; 
\begin{cases}
0 & \bs x\in \ms V^B_N\;, \\
1 & \bs x\in \ms V^A_N  \;, 
\end{cases}
\qquad 
V^A_N(\bs x) \;=\; 
\begin{cases}
V_N^{\bs z}(\bs x) & \bs x\in \mf B^{\bs z}_N \;, \\
(1/2) & \text{otherwise},
\end{cases}
\end{equation*}
where $V_N^{\bs z}$ is the function defined in \eqref{v03}.

\begin{asser}
\label{vs16}
Let $\varepsilon_N$ be a sequence such that $N \varepsilon^3_N \to 0$,
$\exp\{- N \varepsilon^2_N\}$ converges to $0$ faster than any
polynomial. Then,
\begin{equation*}
\frac{Z_N }{(2\pi N)^{d/2}}  \, 2 \pi N\, e^{ N  F (\bs z)}\,  
D_N(V^A_N) \;\le\; [1+o_N(1)] \sum_{\bs z\in \mf S(A)} 
\,  \frac{\mu (\bs z)}{\sqrt{- \det {\rm Hess}\, F(\bs z)}} \;\cdot
\end{equation*}
\end{asser}

\begin{proof} 
We estimate the Dirichlet form of $V^A_N$ inside the sets $\mf B^{\bs
  z}_N$, $\bs z \in \mf S (A)$, at the boundary of $\ms U_N$,
and at the boundary of $B^{\bs z}_N$ which is contained in $\ms U_N$. 

Denote by $\partial \ms U_N$ the outer boundary of $\ms U_N$. The
contribution to the Dirichlet form $D_N(V^A_N)$ of the edges in
$\partial \ms U_N$ is less than or equal to
\begin{equation*}
\sum_{i=1}^d\sum_{\bs x\in\partial\ms U_N}\mu_N(\bs x)
[R_N(\bs x, \bs x + \bs e_i)+R_N(\bs x, \bs x - \bs e_i)]
\leq \frac{C_0}{Z_N}e^{-NF(\bs z)}\sum_{x\in\partial \ms U_N}
e^{-\vartheta N \varepsilon_N^2},
\end{equation*}
where $C_0$ denotes a finite constant which does not depend on $N$ and
whose value may change from line to line. The sum on the right hand
side is bounded by $C_0N^{d-1}e^{-\vartheta N \varepsilon_N^2}$, which
vanishes as $N\uparrow\infty$ in view of our choice of
$\varepsilon_N$.

Let $R^\pm_N(\bs z) = \partial^{\rm in}_\pm B^{\bs z}_N\cap \ms U_N$,
$\bs z \in \mf S(A)$. we estimate the contribution to the Dirichlet
form $D_N(V^A_N)$ of the edges in $R^-_N(\bs z)$, the one of
$R^+_N(\bs z)$ being analogous. By the definition of $V^A_N$ this
contribution is bounded by
\begin{equation}
\begin{split}
\label{v37}
& \sum_{i=1}^d\sum_{\bs x\in R^-_N(\bs z)}\mu_N(\bs x)\,
[R_N(\bs x, \bs x + \bs e_i)+R_N(\bs x, \bs x - \bs e_i)]\,
V^{\bs z}_N(\bs x)^2\;, \\
&\qquad \le\; 
\frac{C_0}{Z_N}e^{-NF(\bs z)}\sum_{\bs x\in R^-_N(\bs z)}
e^{-(1/2)N(\bs y \cdot \bb M^{\bs z} \bs y)}V^{\bs z}_N(\bs x)^2\; ,
\end{split}
\end{equation}
where $\bs y= \bs x - \bs z$. In the remainder of this paragraph we
omit the dependence on~$\bs z$ in the notation. Since $\bs x$ belongs
to $B_N$, $\exp\{-(1/2)N(\bs y \cdot \bb M \bs y)\}$ is less than or
equal to $\exp\{(1/2)\mu N\varepsilon^2_N\}\exp\{-(1/2)N\sum_{2\leq j
  \leq d}\lambda_j (\bs y \cdot \bs w ^j)^2\}$. On the other hand, by
a change of variables, 
\begin{equation*}
V_N(\bs x)^2=\Big(\frac 1{\sqrt{2\pi}} 
\int_{-\infty}^{(N \mu)^{1/2} (\bs y \cdot \bs v)} 
e^{- r^2/2} \, dr \Big)^2\;.
\end{equation*} 
Since $\bs x$ belongs to $\partial^{\rm in}_- B_N$, $\bs y \cdot \bs v
\le - \varepsilon_N + C_0 N^{-1}$. The previous expression is
therefore less than or equal to $(C_0/N\varepsilon_N^2) \exp\{-\mu
N\varepsilon_N^2\}$ because $\int_{(-\infty, A]} \exp\{-(1/2) r^2\} dr$
$\le |A|^{-1} \exp\{-(1/2) A^2\}$ for $A<0$.  This proves that the sum
appearing in \eqref{v37} is less than or equal to $C_0 N^{d-1}
\exp\{-(1/2)\mu N\varepsilon_N^2\}$, which vanishes as
$N\uparrow\infty$, in view of the definition of $\varepsilon_N$.

Since, for each $\bs z \in \mf S(A)$, the set $\mf B^{\bs z}_N$ is
contained in $B_N^{\bs z}$, the contribution to the Dirichlet form of
the bonds in the set $\mf B^{\bs z}_N$ is less than or equal to
$D_N(V_N; B^{\bs z}_N)$. To conclude the proof it remains to recall
Assertion \ref{vs02}.
\end{proof}

\section{Lower bound for the capacities}
\label{vsec01}

We prove in this section the lower bound of Theorem \ref{vs05}.  The
proof is based on the arguments presented in \cite{bbi1, bbi2}.

\begin{proposition}
\label{vs03} 
For every proper subset $A$ of $\{1, \dots, \ell\}$,
\begin{equation*}
\liminf_{N\to\infty}
\frac{Z_N }{(2\pi N)^{d/2}}  \, 2 \pi N\, e^{ N  F (\bs z)}\,  
\Cap_N (\ms E_N(A), \ms E_N(A^c )) \;\ge\;
\sum_{\bs z\in \mf S(A)} 
\,  \frac{\mu (\bs z)}{\sqrt{- \det {\rm Hess}\, F(\bs z)}} \;\cdot
\end{equation*}
\end{proposition}

The idea of the proof is quite simple. It is based on Thomson's
principle \cite[Proposition 3.2.2]{g1} which expresses the inverse of
the capacity as an infimum over divergence free, unitary flows. The
construction of a unitary flow from $\ms E_N(A)$ to $\ms E_N(A^c )$
will be done in two steps. We first construct a unitary flow from $\ms
E_N(A)$ to $\ms E_N(A^c )$ for each saddle point $\bs z\in \mf
S(A)$. Then, we define a unitary flow from $\ms E_N(A)$ to $\ms
E_N(A^c )$ as a convex combination of the unitary flows defined in the
first step.

\smallskip\noindent{\bf Step 1: Flows associated to saddle points.}
The main difficulty of the proof of Proposition \ref{vs03} consists in
defining unitary flows associated to saddle points. Fix $\bs z\in \mf
S(A)$ and two wells $W_a$, $W_b$ such that $a\in A$, $b\in A^c$, $\bs
z\in W_a\cap W_b$. Assume, without loss of generality, that all
coordinates of the vector $\bs v$ are non-negative.  Let $B_N$ be the
subset defined by
\begin{equation*}
B_N \;=\; \Big \{\bs x \in \Xi_N : 
\vert (\bs x - \bs z) \cdot \bs v \vert  \le \varepsilon_N 
\,,\, \max_{2\le j\le d} 
\vert (\bs x - \bs z) \cdot \bs w^j \vert  \le 
\varepsilon_N \, \Big\}\;,
\end{equation*}
where $\varepsilon_N$ is a sequence such that $N \varepsilon^3_N \to
0$, $\exp\{- N \varepsilon^2_N\}$ converges to $0$ faster than any
polynomial. Note that the definition of the set $B_N$ changed with
respect to the one of the previous section.

Keep in mind that we assumed $\bs v$ to be a vector with non-negative
coordinates. Denote by $N(\bs v)$ the set of positive coordinates of
$\bs v$, $N(\bs v) = \{ j : \bs v_j >0\}$. Let $Q^o_N$ be the cone
$Q^o_N = \{\bs x\in N^{-1} \bb Z^d : \bs x_j \ge 0 \,,\, j\in N(\bs v)
\text{ and } \bs x_j = 0 \,,\, j\not\in N(\bs v)\}$, and let $Q^{\bs
  x}_N$, $\bs x \in N^{-1} \bb Z^d$, be the cone $Q^o_N$ translated by
$\bs x$, $Q^{\bs x}_N = \{\bs x + \bs x' : \bs x' \in Q^o_N\}$.

Denote by $\partial^{\rm in}_- B_N$ the inner boundary of $B_N$,
defined as $\partial^{\rm in}_- B_N = \{\bs x\in B_N : \exists \, j
\text{ s.t. } [\bs x - \bs z -\bs e_j] \cdot \bs v <-
\varepsilon_N\}$. Denote by $Q^+_N$ the set of all cones with root in
$\partial^{\rm in}_- B_N$, $Q^+_N = \cup_{\bs x \in \partial^{\rm
    in}_- B_N} Q^{\bs x}_N$, and let
\begin{equation*}
Q_N \;=\; \big\{\bs x\in Q^+_N : (\bs x - \bs z) \cdot \bs v \le
\varepsilon_N \big\}\;.
\end{equation*}
Note that $B_N \subset Q_N$. Figure~\ref{fig4} represents the sets $B_N$, $Q_N$.

\begin{figure}[htb]
  \centering
  \def\svgwidth{150pt}
  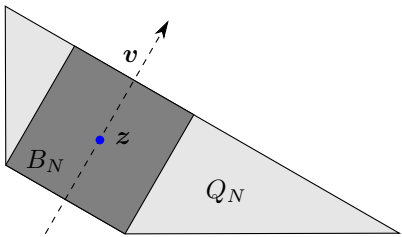
  \caption{The vector $\bs v$, the set $B_N$ in dark gray, and the set
    $Q_N$ in light gray. }
\label{fig4}
\end{figure}

There exists a finite constant $C_0$, independent of $N$, such that
for all $N\ge 1$,
\begin{equation}
\label{v12}
\max_{2\le k\le d} \max_{\bs x\in Q_N} |\, (\bs x - \bs z) \cdot \bs
w^k\,| \;\le\; C_0 \, \varepsilon_N\;. 
\end{equation}
Indeed, if $\bs x$ belongs to $Q_N$, $\bs x = \bs x' + \bs x''$, where
$\bs x'\in \partial^{\rm in}_- B_N$ and $\bs x'' \in Q^o_N$. On the
one hand, $\bs x' \in B_N$ so that $|\, (\bs x' - \bs z) \cdot \bs
w^k\,| \;\le\; \varepsilon_N$ for all $k$ and $N$.  On the other hand,
$\bs x'' \cdot \bs v = [\bs x - \bs z]\cdot \bs v - [\bs x' - \bs
z]\cdot \bs v$. The first term is bounded by $\varepsilon_N$ because
$\bs x$ belongs to $Q_N$. As $\bs x'\in B_N$, the second term is
absolutely bounded by $\varepsilon_N$. This proves that $\bs x''_j \le
C_0\, \varepsilon_N$ for all $j\in N(\bs v)$. The inequality holds
trivially for $j\not \in N(\bs v)$ from what we conclude that there
exists $C_0$ such that $\bs x''_j \;\le\; C_0 \, \varepsilon_N$ for
all $j$ and $N$. Assertion \eqref{v12} follows from this bound and
from the bounds obtained on $\bs x'$. \smallskip

Denote by $\partial B_N$ the external boundary of the set $B_N$, the
set of sites which do not belong to $B_N$ and which have a neighbor in
$B_N$: $\partial B_N = \{\bs x\not \in B_N : \exists\, j \text{ s.t. }
\bs x + \bs e_j \text{ or } \bs x - \bs e_j \in B_N\}$.  Two pieces of
the external boundary of $B_N$ play an important role in the proof of
the lower bound for the capacity. Denote by $\partial_\pm B_N$ the
sets
\begin{equation*}
\partial_- B_N = \Big \{\bs x \in \partial B_N : 
(\bs x - \bs z) \cdot \bs v < - \varepsilon_N  \Big\} \;, 
\quad \partial_+ B_N = \Big \{\bs x \in  \partial  B_N : 
(\bs x - \bs z) \cdot \bs v > \varepsilon_N \, \Big\} \;.
\end{equation*}

Denote by $\partial_+ Q_N$ the outer boundary of $Q_N$ defined by
$\partial_+ Q_N = \{\bs x \in \Xi_N : [\bs x - \bs z] \cdot \bs v >
\varepsilon_N \text{ and } \exists \, j \text{ s.t. } \bs x -\bs e_j
\in Q_N\}$.  We shall construct a divergence free, unitary flow from
$\ms E^a_N$ to $\partial_- B_N$, one from $\partial_- B_N$ to
$\partial_+ Q_N$ and a third one from $\partial_+ Q_N$ to $\ms
E^b_N$. The more demanding one is the flow from $\partial_- B_N$ to
$\partial_+ Q_N$.

\smallskip\noindent{\bf 1.A. Sketch of the proof.}  To explain the
idea of the proof of this part, we first consider the case where the
eigenvector $\bs v$ associated to the negative eigenvalue of
$({\rm Hess}\, F)(\bs z)$ is $\mf e_1$, the first vector of the
canonical basis. In this case the cone $Q^o_N$ introduced in the
previous section is just a ``straight line'': $Q^o_N = \{(k/N, 0,
\dots, 0) : k\ge 0\}$ and $\partial_+ Q_N$ coincides with $\partial_+
B_N$.

We know that the optimal unitary flow from $\partial_- B_N$to
$\partial_+ B_N$ is given by $\widehat{\Phi} (\bs x, \bs y) = c(\bs x,
\bs y) [\widehat{V} (\bs x) - \widehat{V}(\bs y)]/\Cap (\partial_-
B_N, \partial_+ B_N)$, where $c(\bs x, \bs y) = \mu_N(\bs x) \, R_N
(\bs x, \bs y)$ is the conductance between the vertices $\bs x$ and
$\bs y$ and $\widehat{V}$ is the equilibrium potential between
$\partial_- B_N$ and $\partial_+ B_N$.  We introduced in \eqref{v03}
an approximation $V$ of the equilibrium potential $\widehat{V}$. A
calculation shows that the flow $\Phi(\bs x, \bs y) = c(\bs x, \bs y)
[V(\bs x) - V(\bs y)]$ is almost constant along the $\bs v$
direction. Hence, in the case where $\bs v = \mf e_1$, a natural
candidate is a flow constant along the $\mf e_1$ direction.  Denote a
point $\bs x\in\Xi_N$ as $(\hat {\bs x},\check {\bs x})$ where $\hat
{\bs x}\in N^{-1} \bb Z$ and $\check {\bs x}\in N^{-1} \bb Z^{d-1}$,
and let $\check{B}_N = \{\check {\bs x} \in N^{-1} \bb Z^{d-1} :
\exists\, x\in N^{-1} \bb Z \text{ s.t. } (x,\check {\bs x}) \in
B_N\}$,
\begin{equation*}
\Phi(\bs x, \bs y) \;=\; 
\begin{cases}
\Phi(\check {\bs x}) & \text{if $\bs y = \bs x + \bs e_1$, $\bs x \in
  B_N \cup \partial_- B_N$,} \\
0 & \text{otherwise,}
\end{cases}
\end{equation*}
where $\Phi : \check{B}_N \to \bb R_+$ is such that $\sum_{\bs x\in
  \check{B}_N} \Phi(\check {\bs x}) =1$.

By Thomson's principle, the inverse of the capacity is bounded above by
the energy dissipated by the flow $\Phi$:
\begin{equation}
\label{v20}
\frac 1{\Cap_N (\partial_- B_N, \partial_+ B_N)} \;\le\; \Vert
\Phi\Vert^2 \;:=\;
\sum_{\bs x \in B_N \cup \partial_- B_N} \frac 1{c(\bs x, \bs x + \bs e_1)} \, 
\Phi(\bs x, \bs  x+ \bs e_1)^2\;,
\end{equation}
By definition of the flow and by a second order Taylor expansion, the
previous sum is equal to
\begin{equation*}
[1+ o_N(1) ] \, Z_N \, e^{NF(\bs z)}
\sum_{\bs x \in B_N \cup \partial_- B_N} e^{(N/2) \, (\bs y \cdot \bb M \bs y)}
\, \Phi(\check{\bs x})^2\;,
\end{equation*}
provided $N \varepsilon^3_N \to 0$. In this equation, $\bs y = \bs x -
\bs z$. Recall from \eqref{v04} the definition of the matrices $\bb
V$, $\bb D$. Let $\check{\bb D}$ be the diagonal matrix in which the
entry $\lambda_1=-\mu$ has been replaced by $0$, and let $\check{\bb
  M}$ be the symmetric matrix $\check{\bb M} = \bb V \check{\bb D} \bb
V^*$. In particular, for any vector $\bs y$, $\bs y \cdot \check{\bb
  M} \bs y = \sum_{2\le k\le d} \lambda_k (\bs y \cdot \bs w^k)^2$,
and $\bs y \cdot \bb M \bs y = \bs y \cdot \check{\bb M} \bs y - \mu
\hat{\bs y}^2$. With this notation, and since $\bs y \cdot \check{\bb
  M} \bs y$ depends on $\bs y$ only as a function of $\check{\bs y}$,
we may rewrite the previous sum as
\begin{equation*}
[1+ o_N(1) ] \, Z_N \, e^{NF(\bs z)}
\sum_{\check{\bs x}\in \check{B}_N} 
e^{(N/2) \, (\bs y \cdot \check{\bb M} \bs y)}
\, \Phi(\check{\bs x})^2 \sum_{k} e^{- \mu (N/2) k^2} \;,  
\end{equation*}
where the second sum is performed over all $k\in N^{-1}\bb Z$ such
that $-\varepsilon_N - N^{-1}\le k \le \varepsilon_N$.  The optimal
choice of $\Phi$ satisfying $\sum_{\bs x\in \check{B}_N} \Phi(\check
{\bs x}) =1$ is
\begin{equation*}
\Phi(\check {\bs x}) \;=\; e^{- (N/2) \, (\bs y \cdot
  \check{\bb M} \bs y)}/\sum_{\bs x\in \check{B}_N} e^{- (N/2)
  \, (\bs y \cdot \check{\bb M} \bs y)}\;.
\end{equation*}
With this choice the previous sum becomes
\begin{equation*}
[1+ o_N(1) ] \, Z_N \, e^{NF(\bs z)}
\frac{\sum_{k} e^{- \mu (N/2) k^2}}
{\sum_{\check{\bs x}\in \check{B}_N} 
e^{- (N/2) \, (\bs y \cdot \check{\bb M} \bs y)}} \;\cdot
\end{equation*}
At this point we may repeat the arguments presented at the end of the
proof of Assertion \ref{vs02} to conclude that the previous expression
is equal to
\begin{equation*}
[1+ o_N(1) ] \, Z_N \, e^{NF(\bs z)}
\frac{(2\pi N) \sqrt{- \det \bb M/\mu}}
{(2\pi N)^{d/2} \sqrt{\mu}} \;=\;
[1+ o_N(1) ] \, Z_N \, e^{NF(\bs z)}
\frac{(2\pi N) \sqrt{- \det \bb M}}
{(2\pi N)^{d/2} \mu} \;,
\end{equation*}

In conclusion, we constructed a divergence free, unitary flow $\Phi$
from $\partial_- B_N$ to $\partial_+ B_N$ whose dissipated energy,
$\Vert \Phi\Vert^2$, defined in \eqref{v20} satisfies
\begin{equation*}
\lim_{N\to\infty} \frac{(2\pi N)^{d/2}}{Z_N} \, 
\frac 1{2\pi N} \, e^{-NF(\bs z)} \, \Vert \Phi\Vert^2
\;=\; \frac{ \sqrt{- \det [({\rm Hess} \, F)(\bs z)] }} {\mu}  
\;\cdot
\end{equation*}

\smallskip\noindent{\bf 1.B. A unitary flow from $\partial_- B_N$ to
  $\partial_+ Q_N$.} We turn now to the general case. We learned from
the previous example that the optimal flow is $\Phi (\bs x, \bs y) =
M^{-1}_N \, c (\bs x, \bs y) [V(\bs x) - V(\bs y)]$, where $V$ is the
function introduced in \eqref{v03} and $M_N$ a constant which turns
the flow unitary. We thus propose the flow
\begin{equation}
\label{v07}
\Phi (\bs x, \bs x + \bs e_j) \;=\; 
\frac {\sqrt{ - \det \bb M/\mu}}{(2\pi N)^{(d-1)/2}}\, 
\bs v_j \, e^{-(N/2) \bs y \cdot \check{\bb M} \bs y}\;.
\end{equation}

We claim that $\Phi$ is an essentially unitary flow:
\begin{equation}
\label{v08}
\sum_{j=1}^d \sum_{\bs x\in \partial_{j,-} B_N} 
\Phi (\bs x, \bs x + \bs e_j) \;=\; [1+o_N(1)] \;,
\end{equation}
where $\partial_{j,-} B_N$ represents the set of points $\bs x \in
\partial_-B_N$ such that $\bs x + \bs e_j\in B_N$.  We have to show
that
\begin{equation}
\label{v09}
\sum_{j=1}^d \bs v_j \sum_{\bs x\in \partial_{j,-} B_N} e^{-(N/2) \bs y
  \cdot \check{\bb M} \bs y} \;=\; [1+o_N(1)]\,
\frac {(2\pi N)^{(d-1)/2}\, \sqrt{\mu}} {\sqrt{ - \det \bb M}}\;\cdot
\end{equation}
Fix $1\le j\le d$, and let $V = \{\bs x \in \bb R^d : [\bs x - \bs z]
\cdot \bs v = - \varepsilon_N\}$.  Denote by $\delta(\bs x)$, $\bs
x\in \partial_{j,-} B_N$, the amount needed to translate $\bs x$ in
the $\bs e_j$-direction for $\bs x$ to belong to $V$: $\bs x +
\delta(\bs x) \bs e_j \in V$. Observe that $\delta (\bs x)\in (0,1]$.
Let $T(\bs x) = \bs x + \delta(\bs x) \bs e_j$, $\bs
x\in \partial_{j,-} B_N$.  Since $\delta(\bs x)$ is absolutely bounded by
$1$,
\begin{equation*}
\sum_{\bs x\in \partial_{j,-} B_N} e^{-(N/2) \bs y
  \cdot \check{\bb M} \bs y}\;=\; [1+o_N(1)]\, 
\sum_{\bs x\in \partial_{j,-} B_N} \exp\big\{ -(N/2) \sum_{k=2}^d
\lambda_k \{ [T(\bs x) - \bs z] \cdot \bs w^k \}^2 \big\}
\end{equation*}
Replacing $\bs x$ by $\sqrt{N} \bs x$, and approximating the sum
appearing on the right hand side by a Riemann integral, the previous
term becomes
\begin{equation*}
\begin{split}
& [1+o_N(1)]\, \bs v_j\, N^{(d-1)/2} \prod_{k=2}^d 
\int_{-\sqrt{N} \varepsilon_N}^{\sqrt{N} \varepsilon_N}
e^{ -(1/2) \lambda_k r^2}\, dr \\
&\quad =\;
[1+o_N(1)]\, \bs v_j\, (2\pi N)^{(d-1)/2} 
\frac {\sqrt{\mu}}{\sqrt{ - \det \bb M}}\;,
\end{split}
\end{equation*}
where $\bs v_j$ appeared to take into account the tilt of the
hypersurface $V$. Multiplying the last term by $\bs v_j$ and summing
over $j$ we get \eqref{v09} because $\Vert \bs v\Vert =1$. This proves
that the flow $\Phi$ is essentially unitary, as stated in
\eqref{v08}. 

\smallskip\noindent{\bf 1.C. Turning the flow divergence free.}  In
this subsection, we add a correction $R$ to the flow $\Phi$ to turn it
divergence free.  We start with an estimate on the divergence of the
flow $\Phi$.  Denote by $(\text{div } \Phi) (\bs x)$ the divergence of
the flow $\Phi$ at $\bs x$:
\begin{equation*}
(\text{div } \Phi) (\bs x) \;=\; \sum_{j=1}^d 
\{ \Phi(\bs x, \bs x + \bs e_j) - \Phi(\bs x- \bs e_j,
\bs x)\}\;.
\end{equation*}
We claim that there exists a finite constant $C_0$, independent of
$N$, such that
\begin{equation}
\label{v11}
\max_{i\in N(\bs v)} \max_{\bs x\in Q_N} \Big| \, \frac{(\text{div }
  \Phi) (\bs x)} {\Phi(\bs x, \bs x + \bs e_i)}  \,\Big| 
\;\le \; C_0 \varepsilon^2_N\;.
\end{equation}

Fix $i\in N(\bs v)$, $\bs x\in Q_N$, and recall the definition of the
flow $\Phi$. By \eqref{v12}, by definition of the matrix $\check{\bb
  M}$ and by a second order Taylor expansion, for each $1\le i\le d$,
\begin{equation*}
\begin{split}
\sum_{j=1}^d \frac{\Phi(\bs x, \bs x + \bs e_j) - \Phi(\bs x- \bs
e_j, \bs x)}{\Phi(\bs x, \bs x + \bs e_i)} \; &=\;
\sum_{j=1}^d \frac{\bs v_j}{\bs v_i} \sum_{k=2}^d \lambda_k (\mf e_j
\cdot \bs w^k) ([\bs x - \bs z] \cdot \bs w^k) + O(\varepsilon^2_N) \\
& =\;
\frac{1}{\bs v_i} \sum_{k=2}^d \lambda_k (\bs v \cdot \bs w^k) 
([\bs x - \bs z] \cdot \bs w^k) + O(\varepsilon^2_N) \;.    
\end{split}
\end{equation*}
The first term on the right hand side vanishes because $\bs v$ is
orthogonal to $\bs w^k$, which proves \eqref{v11}. \smallskip

We now define a correction $R$ to the flow $\Phi$ to turn it
divergence free.  Let $G_0 = \partial_- B_N$, $G$ for generation.
Define recursively the sets $G_{k}$, $k\ge 1$, by
\begin{equation*}
G_{k+1} \;=\; \Big\{ \bs x\in Q_N : \bs
x-\bs e_j \in \bigcup_{\ell=0}^k G_\ell \cup Q^c_N
\text{ for all } j \in N(\bs v) \Big\}\;, \quad k\ge 0\;.
\end{equation*}
The first three generations are illustrated in Figure~\ref{fig5}.
Denote by $K_N$ the smallest integer $k$ such that $Q_N \subset
\cup_{1\le \ell\le k} G_\ell$. Clearly, $K_N \le C_0
\varepsilon^{-1}_N$ for some finite constant $C_0$.

\begin{figure}[htb]
  \centering
  \def\svgwidth{300pt}
  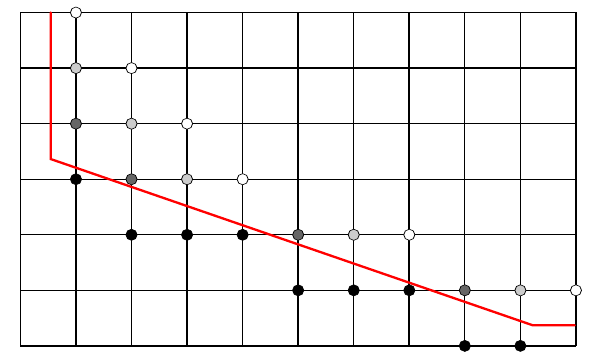
  \caption{The first three generations. The red line represents the
    boundary of the set~$Q_N$, and the numbers the generation of each
    point.}
\label{fig5}
\end{figure}


The flow $R$ is also defined recursively.  For all $\bs x\in G_1$,
define $R(\bs x- \bs e_j, \bs x)=0$, $1\le j\le d$, and let
\begin{equation}
\label{v14}
R(\bs x, \bs x + \bs e_j) \;=\; p_j \Big\{ 
\sum_{i=1}^d R(\bs x- \bs e_i, \bs x) \;-\;
(\text{div } \Phi) (\bs x) \Big\}\;,
\end{equation}
where $p_j= \bs v_j/\sum_{1\le i\le d} \bs v_i$. Note that $R(\bs x,
\bs x + \bs e_j) = 0$ if $j\not\in N(\bs v)$ and that we may restrict
the sum over $i$ to the set $N(\bs v)$. On the other hand, by
construction, $(\text{div } R) (\bs x) = - (\text{div } \Phi) (\bs x)$
for all $\bs x\in \cup_{1\le \ell\le K_N} G_\ell$.

There exists a finite constant $C_0$ such that 
\begin{equation}
\label{v13}
\max_{i\in N(\bs v)} \max_{\bs x\in G_k} \Big| \, 
\frac{ R (\bs x, \bs x + \bs e_i)} {\Phi(\bs x, \bs x + \bs e_i)}
\,\Big| \;\le \; C_0 \, k \, \varepsilon^2_N
\end{equation}
for all $1\le k\le K_N$.  This assertion is proved by induction. Since
$R(\bs x- \bs e_j, \bs x)=0$ for $\bs x\in G_1$, by \eqref{v11},
$\max_{i\in N(\bs v)} \max_{\bs x\in G_1} | \, R (\bs x, \bs x + \bs
e_i)/\Phi(\bs x, \bs x + \bs e_i) \,| \le C_1 \, \varepsilon^2_N$,
where $C_1$ is the constant $C_0$ appearing on the right hand side of
\eqref{v11}.

Suppose that $\max_{i\in N(\bs v)} \max_{\bs x\in G_j} | \, R (\bs x,
\bs x + \bs e_i)/\Phi(\bs x, \bs x + \bs e_i)| \le C_j \,
\varepsilon^2_N$ for all $1\le j\le k$, where $C_j$ is an increasing
sequence. Fix $i\in N(\bs v)$ and $\bs x\in G_{k+1}$. By definition of
$R$, by \eqref{v11}, and by the induction hypothesis,
\begin{equation*}
\Big| \, \frac{ R (\bs x, \bs x + \bs e_i)} {\Phi(\bs x, \bs x + \bs e_i)}
\,\Big| \;\le C_k \, \varepsilon^2_N \, p_i \sum_{j\in N(\bs v)} 
\frac{ \Phi (\bs x-\bs e_j, \bs x)} {\Phi(\bs x, \bs x + \bs e_i)}
\;+\; C_0 \, \varepsilon^2_N \;.
\end{equation*}
The computations performed to prove \eqref{v11} yield that the first
term on the right hand side is bounded by
\begin{equation*}
\begin{split}
& C_k \, \varepsilon^2_N \, p_i \sum_{j\in N(\bs v)} \frac{\bs v_j}{\bs
  v_i} \Big\{ 1 +\sum_{m=2}^d \lambda_m (\mf e_j \cdot \bs w^m) (\bs y
\cdot \bs w^m)  + C_0 \varepsilon^2_N \Big\}  \\
& \quad = \; C_k \, \varepsilon^2_N \, \big\{ 1 + C_0 \varepsilon^2_N
\big\} \;\le\; C_k \, \varepsilon^2_N \, e^{C_0 \varepsilon^2_N} \;.
\end{split}
\end{equation*}
The identity has been derived using the definition of $p_j$, the
orthogonality of $\bs v$ and $\bs w^k$, and summing first over $j$.
We have thus obtained the recursive relation $C_{k+1} \le [C_0 + C_k
e^{C_0 \varepsilon^2_N}]$ from which it follows that $C_k \le C_0 k
e^{C_0 k \varepsilon^2_N}$. Since $k\le K_N \le C_0
\varepsilon^{-1}_N$, \eqref{v13} holds.

\smallskip\noindent{\bf 1.D. A divergence free unitary flow.}  We
construct in this subsection a divergence-free, unitary flow from
$\partial_- B_N$ to $\partial_+ Q_N$ whose energy dissipated is given
by the right hand side of \eqref{v15}.

Let $\Psi$ be the flow from $\partial_- B_N$ to $\partial_+ Q_N$
defined by $\Psi = \Phi + R$, where $\Phi$ is introduced in
\eqref{v07} and $R$ in \eqref{v14}. By \eqref{v08} and by construction
of $R$, $\Psi$ is a unitary flow. Since $(\text{div } R) (\bs x) = -
(\text{div } \Phi) (\bs x)$ for all $\bs x\in \cup_{0\le \ell\le K_N}
G_\ell$, $\Psi$ is divergence-free. It remains to show that the energy
dissipated by $\Psi$ satisfies
\begin{equation}
\label{v15}
\begin{split}
& \sum_{j=1}^d \sum_{\bs x} \frac 1{c(\bs x, \bs x + \bs e_j)} \,
\Psi(\bs x, \bs x+\bs e_j)^2 \\
&\quad \;=\; [1+ o_N(1)]\, \frac {Z_N}{(2\pi N)^{d/2}} \, 2\pi N
\, e^{N F(\bs z)}\, \frac{ \sqrt{- \det [({\rm Hess} \, F)(\bs z)] }} {\mu} 
\;\cdot
\end{split}
\end{equation}

A second order expansion of $F(\bs x)$ at $\bs z$ taking advantage of
\eqref{v12} and of the fact that $N \varepsilon_N^3 \to 0$ permits to
write the left hand side of the previous equation as
\begin{equation*}
[1+ o_N(1)]\, Z_N\, e^{N F(\bs z)} \sum_{j=1}^d \sum_{\bs x}
e^{(N/2) (\bs y \cdot \bb M \cdot \bs  y)}  \, \Psi(\bs x, \bs x+\bs e_j)^2\;,
\end{equation*}
where, as before, $\bs y = \bs x - \bs z$. We may bound $\Psi(\bs x,
\bs x+\bs e_j)^2$ by $(1+\varepsilon_N) \Phi(\bs x, \bs x+\bs e_j)^2 +
(1+\varepsilon^{-1}_N) R(\bs x, \bs x+\bs e_j)^2$, and apply
\eqref{v13} together with the fact that $k\le K_N \le C_0
\varepsilon_N^{-1}$ to estimate the previous sum by $[1+
O(\varepsilon_N)] \Phi(\bs x, \bs x+\bs e_j)^2$. The previous displayed
equation is therefore equal to the same sum with $\Psi$ replaced by
$\Phi$. Replacing $\Phi(\bs x, \bs x+\bs e_j)$ by its value
\eqref{v07} the previous sum becomes 
\begin{equation*}
[1+ o_N(1)]\, Z_N\, e^{N F(\bs z)} \,
\frac 1{(2\pi N)^{d-1}}\,
\frac{ \det\bb M}{-\mu}
\sum_{j=1}^d \bs v_j^2 \sum_{\bs x}
e^{-(N/2) (\bs y \cdot \bb M_\star \cdot \bs y)}  \;,
\end{equation*}
where $\bb M_\star$ is the matrix introduced in \eqref{v04}. At this
point it remains to recall that $\bs v$ is a normal vector and to
repeat the calculations performed in the proof of the upper bound of
the capacity to retrieve \eqref{v15}.

\smallskip\noindent{\bf 1.E. A unitary flow from $\ms E^a_N$ to
  $\partial_- B_N$.} We extend in this section the flow $\Psi$ from
$\partial_- B_N$ to $\ms E^a_N$. The same arguments permit to extend
the flow $\Psi$ from $\partial_+ Q_N$ to $\ms E^b_N$.  The idea is
quite simple. For each bond $(\bs x, \bs x + \bs e_j)$, $\bs
x\in \partial_- B_N$, $\bs x + \bs e_j \in B_N$, we construct a path
of nearest neighbor sites $(\bs x = \bs x^0, \bs x^1, \dots, \bs
x^n)$, $\bs x^n\in \ms E^a_N$, from $\bs x$ to $\ms E^a_N$, and we
define the flow $\Psi_{\bs x, \bs e_j}$ from $\bs x$ to $\ms E^a_N$ by
$\Psi_{\bs x, \bs e_j} (\bs x^k,\bs x^{k+1}) = - \Psi (\bs x, \bs x +
\bs e_j)$. Adding all flows $\Psi_{\bs x, \bs e_j}$ we obtain a
divergence free, unitary flow from $\partial_- B_N$ to $\ms E^a_N$
whose dissipated energy is easily estimated.

We start defining the paths.  For $\bs y\in\bb R^d$, denote by $[\bs
y]$ the vector whose $j$-th coordinate is $[\bs y_j N]/N$, where $[a]$
stands for the largest integer less than or equal to $a\in\bb R$.  Fix
$\bs x \in \partial_- B_N$. Denote by $\bs x(t)$ the solution of the
ODE $\dot {\bs x}(t) = - \nabla F(\bs x(t))$ with initial condition
$\bs x(0) = \bs x$. Since $[\bs x-\bs z]\cdot \bs v <0$, $\bs x(t)$
converges, as $t\to\infty$, to one of the local minima of $F$ in
$W_a$.  Let $T = \inf\{t>0 : \bs x(t) \in W^o_a\}$, where $W^o_a$ is
an open set whose closure is contained in $W^\epsilon_a$, the set
introduced in \eqref{v16}. Let $\bs y^0 = \bs x, \bs y^1, \dots, \bs
y^m$ be the sequence of points in $\Xi_N$ visited by the trajectory
$[\bs x(t)]$, $0\le t\le T$. If necessary, add points to this sequence
in order to obtain a sequence $\bs x^0 = \bs x, \bs x^1, \dots, \bs
x^{m'}$ such that $\Vert \bs x^k - \bs x^{k+1}\Vert = N^{-1}$. Remove
from this sequence the loops and denote by $n$ the length of the
path. Since $F(\bs x(t))$ does not increase in time, and since for all
$k$ there exists some $0\le t\le T$ such that $\Vert \bs x^k - \bs
x(t)\Vert \le d/N$, there exists a finite constant $C_0$ such that
\begin{equation}
\label{v17}
F(\bs x^k) \;\le\; F(\bs x) \;+\; \frac{C_0}N \text{ for all } 0\le
k\le n\;. 
\end{equation}

Fix a bond $(\bs x, \bs x + \bs e_j)$, $\bs x\in \partial_- B_N$, $\bs
x + \bs e_j \in B_N$. Define the flow $\Psi_{\bs x, \bs e_j}$ from
$\bs x$ to $\ms E^a_N$ by $\Psi_{\bs x, \bs e_j} (\bs x^k,\bs x^{k+1})
= - \Psi (\bs x, \bs x + \bs e_j)$, $0\le k<n$. We claim that there
exists a finite constant $C_0$ and a positive constant $c_0$ such that
\begin{equation}
\label{v18}
\Vert \Psi_{\bs x, \bs e_j} \Vert^2 \;\le\; C_0 N Z_N e^{N F(\bs z)} 
e^{- c_0 N \varepsilon^2_N}\;.
\end{equation}
The proof of this assertion is simple.  Since $\Psi (\bs x, \bs x +
\bs e_j) = \Phi (\bs x, \bs x + \bs e_j)$ is given by \eqref{v07}, by
\eqref{v17},
\begin{equation*}
\Vert \Psi_{\bs x, \bs e_j} \Vert^2 \;\le\;
C_0 Z_N \sum_{k=0}^{n-1} e^{N F(\bs x^k)} \Phi (\bs x, \bs x + \bs e_j)^2\;\le\;
\frac{C_0  Z_N n}{N^{d-1}} e^{N F(\bs x)}  e^{-N (\bs y
  \cdot \check{\bb M} \bs y)} \;.
\end{equation*}
By a second order Taylor expansion, $\exp N \{ F(\bs x) - (\bs
  y \cdot \check{\bb M} \bs y) \}$ is less than or equal to
$C_0 \exp \{ N F(\bs z)\} \exp \{ -(1/2) \mu N \varepsilon^2_N\} $
because $N \varepsilon^3_N \to 0$ and $[\bs x-\bs z]\cdot \bs v < -
\varepsilon_N$. This proves \eqref{v18} because $n\le |\Xi_N|$.

Let $\Psi = \sum_{\bs x, j} \Psi_{\bs x, \bs e_j}$, where the sum is
carried over all $\bs x$, $j$ such that $\bs x\in \partial_- B_N$,
$\bs x + \bs e_j \in B_N$. $\Psi$ is a unitary, divergence free flow
from $\partial_- B_N$ to $\ms E^a_N$. Moreover, by Schwarz inequality
and by \eqref{v18},
\begin{equation*}
\Vert \Psi \Vert^2 \;\le\; M\, \sum_{\bs x, j} \Vert\Psi_{\bs x, \bs
  e_j}\Vert^2 \;\le\; C_0 N^{d+1} Z_N e^{N F(\bs z)} 
e^{- c_0 N \varepsilon^2_N}\;, 
\end{equation*}
where $M$ represents the number of flows $\Psi_{\bs x, \bs e_j}$.

Choosing $\varepsilon_N$ appropriately and juxtaposing the flow just
constructed with the one obtained in Section 1.D and a flow from
$\partial_+ Q_N$ to $\ms E^b_N$, similar to the one described in this
section, yields a divergence free, unitary flow from $\ms E^a_N$ to
$\ms E^b_N$, denoted by $\Phi_z$, such that
\begin{equation}
\label{v19}
\lim_{N\to\infty}
\frac {(2\pi N)^{d/2}}{Z_N }  \, 
\frac 1 {2\pi N}\, e^{ - N  F (\bs z)}\,
\Vert \Phi_z \Vert^2 \;=\; 
\frac {\sqrt{- \det [({\rm Hess}\, F) (\bs z)]}}{\mu (\bs z)}\;.
\end{equation}

\smallskip\noindent{\bf Step 2. Conclusion.}  Up to this point, for
each saddle point $z$ separating $\ms E_N(A)$ from $\ms E_N(A^c )$ we
constructed a divergence free, unitary flow $\Phi_z$ from $\ms E_N(A)$
to $\ms E_N(A^c )$ for which \eqref{v19} holds.  Denote the right hand
side of \eqref{v19} by $a(z)$ and observe that $F (\bs z)$ is constant
for $\bs z\in\mf S(A)$.

Let $\Phi$ be a convex combination of the previous flows: $\Phi =
\sum_{\bs z\in\mf S(A)} \theta_{\bs z} \Phi_{\bs z}$, where
$\theta_{\bs z}\ge 0$, $\sum_{\bs z\in\mf S(A)} \theta_{\bs z} =1$. By
construction, $\Phi$ is a flow from $\ms E_N(A)$ to $\ms E_N(A^c
)$. On the other hand, since the saddle points are isolated and since
the main contribution of the flow $\Phi_z$ occurs in a small
neighborhood of $z$
\begin{equation*}
\limsup_{N\to\infty}
\frac {(2\pi N)^{d/2}} {Z_N }  \, 
\frac 1{2\pi N}\, e^{ - N  F (\bs z)}\,
\Vert \Phi \Vert^2 \;\le\; 
\sum_{\bs z\in\mf S(A)} \theta^2_{\bs z} \, a (\bs z)\;.
\end{equation*}
The optimal choice for $\theta$ is $\theta_z = a (\bs
z)^{-1}/\sum_{\bs z'} a (\bs z')^{-1}$. With this choice the right
hand side of the previous equation becomes $(\sum_{\bs z\in\mf S(A)} a
(\bs z)^{-1})^{-1}$. Proposition \ref{vs03} follows from Thomson's
principle and from the previous bound for the flow $\Phi$.  \qed

\section{Proof of Theorem \ref{vs04}}
\label{vsec03}

Theorem \ref{vs04} follows from Propositions \ref{vs17} and \ref{vs18}
below. Throughout this section $1\le i\le i_0$ and $1\le j\le \ell_i$
are fixed and dropped from the notation. 

\begin{proposition}
\label{vs17}
For every disjoint subsets $A$, $B$ of $S$,
\begin{equation*}
\Cap_N (\ms E_N(A), \ms E_N(B) ) \; \le \; [1+o_N(1)]\,  
\frac{(2\pi N)^{d/2}}{Z_N }\,
\frac{e^{- N  H_i}}{2 \pi N} \, \Cap_{\bb G} (A,B)\;.
\end{equation*}
\end{proposition}

The proof of this proposition is similar to the one of Proposition
\ref{vs01} up to Assertion \ref{vs15}. Denote by $\mf S_{i,j}$ the set
of all saddle points in $\Omega^i_j$, and recall the definition of the
set $\ms U_N$ introduced right after Assertion \ref{vs15}.  Let $\mf
B^{\bs z}_N = \ms U_N\cap \ms B^{\bs z}$, $\bs z \in \mf S_{i,j}$,
$\ms V_N = \ms U_N \setminus \cup_{\bs z \in \mf S_{i,j}} \mf B^{\bs
  z}_N$ so that
\begin{equation*}
\ms U_N \;=\; \ms V_N \,\cup\, \bigcup_{\bs z \in \mf S_{i,j}} 
\mf B^{\bs z}_N\;.
\end{equation*}
In contrast with Section \ref{vsec02}, we define a set $\mf B^{\bs
  z}_N$ around each saddle point $\bs z$.  By Assertion \ref{vs15},
the set $\ms V_N$ is formed by several connected components separated
by the sets $\mf B^{\bs z}_N$, $\bs z \in \mf S_{i,j}$. Let $\ms V^a_N$ be
the connected component of $\ms V_N$ which contains a point in $W_a$,
$a\in S$.

Fix two disjoint subsets $A$, $B$ of $S$ and denote by $V_{A,B}$ the
equilibrium potential between $A$ and $B$ for the graph $\bb G$.  Fix
a saddle point $\bs z\in \mf S_{i,j}$ and assume that $\bs z \in
W_a\cap W_b$. Recall the definition of the function $V^{\bs z}_N$
introduced in \eqref{v03} and assume without loss of generality that
$\partial_- B_N \cap W_a \not = \varnothing$ so that $\partial_+B_N
\cap W_b \not = \varnothing$. Define $W^{\bs z}_N : \mf B^{\bs z}_N
\to [0,1]$ as
\begin{equation*}
W^{\bs z}_N  (\bs x) \;=\; V_{A,B} (a) \;+\; [V_{A,B} (b) - V_{A,B}
(a)] \, V^{\bs z}_N  (\bs x)\;.
\end{equation*}
Let $V^{A,B}_N: \Xi_N \to [0,1]$ by
\begin{equation*}
V^{A,B}_N(\bs x) \;=\; 
\begin{cases}
V_{A,B}( a ) & \bs x\in \ms V^{a}_N \;,  \\
W^{\bs z}_N (\bs x) & \bs x\in \mf B^{\bs z}_N \;, \\
(1/2) & \text{otherwise}.
\end{cases}
\end{equation*}

\begin{asser}
\label{vs19} 
Let $\varepsilon_N$ be a sequence such that $N \varepsilon^3_N \to 0$,
$\exp\{- N \varepsilon^2_N\}$ converges to $0$ faster than any
polynomial.  Then,
\begin{equation*}
\frac{Z_N }{(2\pi N)^{d/2}}  \, 2\pi N\, e^{ N  F (\bs z)}\,  
D_N(V^{A,B}_N) \;\le\; [1+o_N(1)] \, D_{\bb G}(V_{A,B})\;,
\end{equation*}
where $D_{\bb G}(V_{A,B})$ represents the Dirichlet form of $V_{A,B}$
with respect to the graph $\bb G$.
\end{asser}

The proof of this assertion is similar to the one of Assertion
\ref{vs16}. Proposition~\ref{vs17} follows from the last assertion and
from the fact that $\Cap_{\bb G} (A,B) = D_{\bb G} (V_{A,B})$.

We conclude the section with the proof of the lower bound.

\begin{proposition}
\label{vs18}
For every disjoint subsets $A$, $B$ of $S$,
\begin{equation*}
\Cap_N (\ms E_N(A), \ms E_N(B) ) \; \ge \; [1+o_N(1)]\,  
\frac{(2\pi N)^{d/2}}{Z_N }\,
\frac{e^{- N  H_i}}{2 \pi N} \, \Cap_{\bb G} (A,B)\;.
\end{equation*}
\end{proposition}

\begin{proof}
Fix two disjoint subsets $A$, $B$ of $S$. We construct below a
divergence-free, unitary flow $\Psi$ from $\ms E_N(A)$ to $\ms
E_N(B)$. 

Recall that we denote by $V_{A,B}$ the equilibrium potential between
$A$ and $B$ in the graph $\bb G$.  Denote by $\varphi = \varphi_{A,B}$
the flow from $A$ to $B$ in the graph $\bb G$ given by $\varphi(a,b) =
\bs c(a,b) [V_{A,B} (a) - V_{A,B} (b)]/\Cap_{\bb G}(A,B)$, and observe
that $\varphi(a,b) = 0$ if $a$, $b$ belong to $A$ or if $a$, $b$
belong to $B$. By \cite[Proposition 3.2.2]{g1},
\begin{equation}
\label{v38}
\frac 1{\Cap_{\bb G}(A,B)} \; =\; \frac 12\, 
\sum_{a,b\in S} \frac 1{\bs c(a,b)}\, \varphi_{A,B}(a,b)^2 \; =:\; 
\Vert \varphi_{A,B} \Vert^2 \;.
\end{equation}

Assume first that each pair of wells has at most one saddle point
separating them, that is, assume that the sets $W_a\cap W_b$ are
either empty or singletons. In this case, each edge $(a,b)$ of the
graph $\bb G$ corresponds to a unique saddle point $\bs z$.

Denote by $\Phi_{a,b}$, $a\not = b\in S$, $\bs c(a,b)>0$, the flow
$\Phi_{\bs z}$ constructed just above \eqref{v19} from $\ms E^a_N$ to
$\ms E^b_N$, where $\bs z\in W_a\cap W_b$ is the saddle point
separating $W_a$ and $W_b$. Note that $\Phi_{a,b} \not = -
\Phi_{b,a}$. We may assume that the flow $\Phi_{a,b}$ is a flow from
$\bs x^a$ to $\bs x^b$, where $\bs x^c$, $c\in S$, are points in $\ms
E^c_N$.  Define the flow $\Psi$ by
\begin{equation*}
\Psi (\bs x, \bs y) \;=\; \sum_{a,b}  \varphi(a,b)
\, \Phi_{a,b} (\bs x, \bs y) \;, \quad \bs x\,,\, \bs y\,\in\,
\Xi_N\;, 
\end{equation*}
where the sum is carried out over all $a\not = b\in S$ such that
$\varphi(a,b)>0$.  We claim that $\Psi$ is a unitary, divergence-free
flow from $\ms E_N(A)$ to $\ms E_N(B)$.

Clearly,
\begin{equation*}
\sum_{\bs x\in \ms E_N(A), \bs y\not \in \ms E_N(A)}
\Psi (\bs x, \bs y) \;=\; \sum_{a,b}  \varphi(a,b)
\sum_{\bs x\in \ms E_N(A), \bs y\not \in \ms E_N(A)}
\, \Phi_{a,b} (\bs x, \bs y) \;.
\end{equation*}
The flows $\Phi_{a,b}$ which cross $\ms E_N(A)$ are the ones starting
or ending at $\ms E_N(A)$. Since, in addition, $\varphi(a,b)=0$ if
$a$, $b\in A$, and $\varphi(a,b)<0$ if $a\not\in A$, $b\in A$, the
previous expression is equal to
\begin{equation*}
\sum_{a\in A ,b\not\in A}  \varphi(a,b)
\sum_{\bs x\in \ms E_N(A), \bs y\not \in \ms E_N(A)}
\, \Phi_{a,b} (\bs x, \bs y) \; =\;
\sum_{a\in A ,b\not\in A}  \varphi(a,b)\;,
\end{equation*}
where the last identity follows from the fact that $\Phi_{a,b}$ is
a unitary flow from $\ms E^a_N$ to $\ms E^b_N$. As $\varphi$ is a
unitary flow from $A$ to $B$, the last sum is equal to $1$, proving
that $\Psi$ is unitary.

To prove that $\Psi$ is divergence-free, fix a site $\bs x \not\in
\{\bs x^c : c\in A \cup B\}$. If $\bs x \not \in\{ \bs x^c : c\in S
\setminus [A\cup B]\}$, $\Psi$ has no divergence at $\bs x$ because it
is the convex combination of flows which have no divergence at $\bs
x$. If $\bs x= \bs x^c$, $c\not\in A \cup B$, the flows $\Phi_{a,b}$,
$a$, $b\not = c$, have no divergence at $\bs x^c$, while the
divergence of $\Phi_{a,c}$ (resp. $\Phi_{c,a}$) at $\bs x^c$ is equal
to $-1$ (resp. $1$) because these flows are unitary and end
(resp. start) at $\bs x^c$. Therefore, the divergence of $\Psi$ at
$\bs x^c$ is equal to
\begin{equation*}
({\rm div}\, \Psi)(\bs x^c) \;=\; \sum_{a,b}  \varphi(a,b)
\, ({\rm div}\, \Phi_{a,b})(\bs x^c) \;=\; - \sum_{a: \varphi(a,c)>0}
\varphi(a,c) \;+\; \sum_{a: \varphi(c,a)>0} \varphi(c,a) \;.
\end{equation*}
Since $\varphi$ is a divergence-free flow in the graph $\bb G$, this
sum vanishes, which proves that $\Psi$ is also divergence-free at $\bs
x^c$, $c\in S \setminus [A\cup B]$.

We claim that the energy dissipated by the flow $\Psi$ is given by
\begin{equation}
\label{v40}
\Vert \Psi \Vert^2 \;=\; [1+o_N(1)]\, 
\frac {Z_N}{(2\pi N)^{d/2}} \, 2\pi N \, e^{ N H_i} 
\frac 1{\Cap_{\bb G}(A,B)} \;.
\end{equation}
Indeed, by definition,
\begin{equation*}
\Vert \Psi \Vert^2 \;=\; \sum_{j=1}^d \sum_{\bs x} \frac 1{c(\bs x, \bs x +
  \bs e_j)} \Psi(\bs x, \bs x + \bs e_j)^2 \;,
\end{equation*}
where the second sum is performed over all $\bs x\in \Xi_N$ such that
$\bs x+\bs e_j\in\Xi_N$. By definition of the flow $\Psi$, the
previous sum is equal to
\begin{eqnarray}
\label{v39}
\!\!\!&\!\!\!&\!\!\!\!
\sum_{a,b}  \varphi(a,b)^2 \sum_{j=1}^d \sum_{\bs x} \frac 1{c(\bs x, \bs x +
  \bs e_j)} \Phi_{a,b} (\bs x, \bs y)^2 \\
\!\!\!&\!\!\!&\!\!\! \;+\; \sum_{(a,b)\not = (a',b')}  \varphi(a,b) \varphi(a',b')
\sum_{j=1}^d \sum_{\bs x} 
\frac 1{c(\bs x, \bs x + \bs e_j)} \Phi_{a,b} (\bs x, \bs x + \bs e_j) 
\Phi_{a',b'} (\bs x, \bs x + \bs e_j) \;. \nonumber
\end{eqnarray}
By \eqref{v15}, the first line is equal to
\begin{equation*}
[1+ o_N(1)]  \, \frac {Z_N}{(2\pi N)^{d/2}} \, 2\pi N \, e^{N H_i} 
\sum_{a,b}  \varphi(a,b)^2
\, \frac{ \sqrt{- \det [({\rm Hess} \, F)(\bs z_{a,b})] }} {\mu(\bs z_{a,b})} \;,
\end{equation*}
where $\bs z_{a,b}$ stands for the saddle point in $W_a\cap W_b$ and
$-\mu(\bs z_{a,b})$ for the negative eigenvalue of $({\rm Hess} \,
F)(\bs z_{a,b})$. By \eqref{v22} and by \eqref{v38}, the previous
sum is equal to
\begin{equation*}
\sum_{a,b}  \frac 1{\bs c(a,b)} \, \varphi(a,b)^2 \;=\; 
\frac 1{\Cap_{\bb G}(A,B)}\;\cdot
\end{equation*}

We turn to the second line of \eqref{v39}. We have seen in the proof
of Proposition~\ref{vs07} that the contribution of the bonds which do
not belong to a mesoscopic neighborhood of the saddle point $\bs
z_{a,b}$ to the total energy dissipated by the flow $\Phi_{a,b}$ is
negligible. We may therefore restrict our attention in the second line
of \eqref{v39} to the points $\bs x$ which belong to one of these
neighborhoods. Since the flow $\Phi_{\bs z}$ vanishes in a
neighborhood of a saddle point $\bs z' \not = \bs z$, the product
$\Phi_{a,b} (\bs x, \bs y) \Phi_{a',b'} (\bs x, \bs y)$ vanishes for
all for $(a,b) \not = (a',b')$ and all $\bs x$ in a neighborhood of
some saddle point $\bs z$. In particular, the second line of
\eqref{v39} is of order 
\begin{equation*}
o_N(1)  \, \frac {Z_N}{(2\pi N)^{d/2}} \, 2\pi N \, e^{N H_i}  \;.
\end{equation*}
Assertion \eqref{v40} follows from the estimates of the two lines of
\eqref{v39}. 

Since $\Psi$ is a divergence-free unitary flow from $\ms E^N(A)$ to
$\ms E^N(B)$, by Thomson's principle, and by \eqref{v40},
\begin{equation*}
\frac 1{\Cap_{N}(\ms E^N(A),\ms E^N(B))} \;\le\; \Vert \Psi \Vert^2
\;=\; [1+o_N(1)]\, 
\frac {Z_N}{(2\pi N)^{d/2}} \, 2\pi N \, e^{ N H_i} 
\frac 1{\Cap_{\bb G}(A,B)} \;\cdot
\end{equation*}
This completes the proof of the proposition in the case where there is
at most one saddle point between two wells.

In the general case, one has to change the definition of $\Psi$ as
follows. For each $a$, $b\in S$ such that $\varphi(a,b)>0$, denote by
$\bs z^1_{a,b}, \dots, \bs z^n_{a,b}$ the set of saddle points between
$W_a$ and $W_b$: $W_a \cap W_b =\{\bs z^1_{a,b}, \dots, \bs
z^n_{a,b}\}$, where $n= n_{a,b}$. Set
\begin{equation*}
\Psi \;=\; \sum_{a,b} \varphi(a,b) \sum_{k=1}^n \, 
\theta_k(a,b) \, \Phi_{\bs z^k_{a,b}}\;, 
\end{equation*}
where the sum is carried out over all $a\not = b\in S$ such that
$\varphi(a,b)>0$, where $\Phi_{\bs z^k_{a,b}}$ is the flow constructed
just above \eqref{v19} from $\ms E^a_N$ to $\ms E^b_N$ passing through
the saddle point $\bs z^k_{a,b}$, and where
\begin{equation*}
\theta_k(a,b) \;=\;  \frac {\mu(\bs z^k_{a,b})}
{\sqrt{- \det [({\rm Hess} \, F)(\bs z^k_{a,b})] }} 
\, \frac 1{\bs c(a,b)} \;.
\end{equation*}
Note that $\sum_k \theta_k(a,b) = 1$. The arguments
presented above for the case where there is at most one saddle point
separating the wells can be easily adapted to the present case.
\end{proof}

\section{Proof of Theorem \ref{vs07}}
\label{vsec05}

According to \cite[Theorem 5.1]{l2}, Theorem \ref{vs07} follows from
Proposition \ref{vs09} below. 

Recall the notation introduced in Section \ref{vsec00}. Fix $1\le i\le
i_0$ and $1\le j\le \ell_i$, which are dropped out from the
notation. Fix a connected component $\varOmega = \varOmega^i_j$, $1\le
m\le n=n_{i,j}$ and denote by $\ms E_{m,N}$ the union of the wells
$\ms E^a_N$, $a\in S_m$, $\ms E_{m,N} = \cup_{a\in S_m} \ms E^a_N$.
As $m$ is fixed throughout this section, it will sometimes be omitted
from the notation.

Denote by $\{T_{m,N}(t) : t\ge 0\}$ the additive functional
\begin{equation*}
T_{m,N}(t) \;=\; \int_0^t \mb 1\{X_N(s) \in \ms E_{m,N}\}\, ds\;,
\end{equation*}
and by $S_{m,N}(t)$ its generalized inverse: $S_{m,N}(t) = \sup\{s\ge
0 : T_{m,N}(s) \le t\}$. The time-change process $X^{m, \rm T}_N(t) :=
X_N(S_{m,N}(t))$ is called the trace process of $X_N(t)$ on $\ms
E_{m,N}$. The process $X^{m, \rm T}_N(t)$ is a $\ms E_{m,N}$-valued,
continuous-time Markov chain. We refer to \cite{bl2} for a summary of
its properties.

Denote by $R^{m, \rm T}_N(\bs x, \bs y)$, $\bs x$, $\bs y\in \ms E_N$,
the jump rates of the trace process. According to \cite[Proposition
6.1]{bl2}, 
\begin{equation*}
R^{m, \rm T}_N(\bs x, \bs y) \;=\; \lambda_N (\bs x) \,
\mb P_{\bs x} \big[\,H^+_{\ms E_{m,N}} = H_{\bs y} \,\big]\,,\quad \bs
x\,,\, \bs y \in \ms E_{m,N}\,,\; \bs x\not=\bs y\,.
\end{equation*}
Denote by $r^m_N(a,b)$ the average rate at which the trace process jumps
from $\ms E^a_N$ to $\ms E^b_N$, $a$, $b\in S_m$:
\begin{equation}
\label{v28}
r^m_N(a,b) \;: =\; \frac 1{\mu_N(\ms E^a_N)} \sum_{\bs x\in \ms E^a_N} 
\mu_N(\bs x) \sum_{\bs y\in \ms E^b_N} R^{m, \rm T}_N(\bs x, \bs y)\;.
\end{equation}

Recall the definition of the projection $\Psi^m_N$ introduced in
\eqref{v25}. Denote by $\bs X^{m, \rm T}_N(t)$ the projection by
$\Psi^m_N$ of the trace process $X^{m, \rm T}_N(t)$, $\bs X^{m, \rm
  T}_N(t) = \Psi^m_N (X^{\rm T}_N(t))$.

\begin{proposition}
\label{vs09}
Fix $1\le i \le i_0$, $1\le j\le \ell_i$, $1\le m\le n_{i,j}$, $a\in
S_m$ and a sequence of configurations $\bs x_N$ in $\ms
E^{a}_N$. Under $\mb P_{\bs x_N}$, the time re-scaled projection of
the trace $\bb X^{m, \rm T}_N(t) = \bs X^{m, \rm T}_N(t \beta_m)$
converges in the Skorohod topology to a $S_m$-valued continuous-time
Markov chain $\bb X^m(t)$ whose jump rates are given by
\eqref{v23}. Moreover, in the time scale $\beta_m$, the time spent by
the original chain $X_N(t)$ outside $\ms E_{m,N}$ is negligible: for all
$t>0$, 
\begin{equation}
\label{v26}
\lim_{N\to\infty} \mb E_{\bs x_N} \Big[ \int_0^t \mb 1\{X_N(s\beta_m) 
\not \in \ms E_{m,N}\}\, ds \Big]\;=\; 0\;.
\end{equation}
\end{proposition}

\begin{proof}
By \cite[Theorem 2.7]{bl2}, the first assertion of the proposition
follows by Lemmata \ref{vs10} and \ref{vs11} below. We turn to the
proof of the second assertion of the proposition.

Fix $\delta >0$ such that $\delta < H_{i+1}-H_i$ and let
$\widetilde{\Omega}^i_\delta = \{\bs x \in \Xi : F(\bs x) \le H_i +
\delta\}$. Denote by $\widetilde{\Omega}_\delta =
\widetilde{\Omega}^{i,j}_\delta$ the connected component which
contains $\Omega^i_j$ and let $A_N = \widetilde{\Omega}_\delta \cap
\Xi_N$.

By the large deviations principle for the chain $X_N(t)$, 
for every $T>0$ and every sequence $\bs x_N\in \ms E_{1,N}$,
\begin{equation*}
\lim_{N\to\infty} \mb P_{\bs x_N} \big[ H_{A^c_N} \le T \beta_m \big]
\;=\; 0\;.
\end{equation*}
This statement can be proved as Theorem 4.2 of Chapter 4, or Theorem
6.2 of Chapter 6 in \cite{fw1}.  It is therefore enough to prove
\eqref{v26} for the chain $X_N(t)$ reflected at $A_N$, the chain
obtained by removing all jumps between $A_N$ and $A^c_N$.

Denote the reflected chain by $\tilde X_N(t)$, by $\tilde\mu_N$ its
stationary state, and by $\tilde{\mb P}_{\bs x}$ the measure on the
path space $D(\bb R_+, A_N)$ induced by the chain $\tilde X_N(t)$
starting from $\bs x\in A_N$. Expectation with respect to $\tilde{\mb
  P}_{\bs x}$ is represented by $\tilde{\mb E}_{\bs x}$. We have to
prove \eqref{v26} with $X_N(t)$, $\mb E_{\bs x_N}$ replaced by $\tilde
X_N(t)$, $\tilde{\mb E}_{\bs x}$, respectively. Equation \eqref{v26}
with these replacements is represented as (\ref{v26}$\,*$).

Let $\Delta_{m,N} = A_N \setminus \ms E_{m,N}$. By definition of the
sets $\ms E^a_N$, for $a\in S_1$, $\tilde\mu_N
(\Delta_{1,N})/\tilde\mu_N (\ms E^a_N)$ is at most of the order
$\exp\{ - N (\theta_1 - \epsilon)\}$, where $\epsilon$ has been
introduced right before \eqref{v16}. For each fixed $1< m\le n$, $a\in
S_m$, $\tilde\mu_N (\Delta_{m,N})/\tilde\mu_N (\ms E^a_N)$ is at most
of the order $\exp\{ - N (\theta_m - \theta_{m-1} -
\epsilon)\}$. Therefore, for every $1\le m\le n$, $a\in S_m$,
\begin{equation}
\label{v27}
\lim_{N\to\infty} \frac{\tilde\mu_N (\Delta_{m,N})}
{\tilde\mu_N (\ms E^a_N)} \;=\; 0\;.
\end{equation}

Fix $1\le m< n$, $a\in S_{m+1}$ and a sequence $\bs x_N\in \ms
E^a_N$. By the large deviations principle for the chain $\tilde
X_N(t)$, for every $T>0$,
\begin{equation}
\label{v29}
\lim_{N\to\infty} \tilde {\mb P}_{\bs x_N} \big[ H_{(\ms E^a_N)^c} 
\le T \beta_m \big] \;=\; 0\;.
\end{equation} 
By \cite[Theorem 2.7]{bl2}, assertion (\ref{v26}$\,*$) follows from
the first part of this proposition and from \eqref{v27}, \eqref{v29},
which concludes the proof.
\end{proof}

Recall that we denoted by $\{\bs m_{a,1}, \dots, \bs m_{a,q}\}$,
$q=q_{a}$, the deepest local minima of $F$ which belong to $W_a$.  

\begin{lemma}
\label{vs10}
Under the hypotheses of Proposition \ref{vs09}, for every $a\in S_m$,
\begin{equation*}
\lim_{N\to \infty} \sup_{\bs y\in \ms E^a_N} 
\frac{\Cap_N(\ms E^a_N, \cup_{b\in S_m, b\not = a} \ms E^b_N)}
{\Cap_N(\{\bs y\} , \{\bs m_{a,1}\})}\;=\;0\; .
\end{equation*}
\end{lemma}

\begin{proof}
Fix $a\in S_m$, $\bs y\in \ms E^a_N$. We estimate $\Cap_N(\{\bs y,
\{\bs m_{a,1}\})$ through Thomson's principle. Let $(\bs y = \bs x_0,
\bs x_1, \dots, \bs x_m = \bs m_{a,1})$ be a path $\gamma$ from $\bs
y$ to $\bs m_{a,1}$ so that $\Vert \bs x_j - \bs x_{j+1}\Vert =
1/N$. By Thomson's principle,
\begin{equation*}
\frac{1}{\Cap_N(\{\bs y\}, \{\bs m_{a,1}\})} \;\le\; \sum_{j=0}^{m-1} 
\frac 1{\mu_N (\bs x_j)\, R_N (\bs x_j , \bs x_{j+1} )}\;\cdot
\end{equation*}
In view of the explicit formulas for the measure $\mu_N$ and the rates
$R_N$, there exists a finite constant $C_0$ such that
\begin{equation*}
\frac 1{\mu_N (\bs x_j)\, R_N (\bs x_j , \bs x_{j+1} )}\; \le\; 
C_0\, Z_N \, e^{N F (\bs x_j)}\;.
\end{equation*}
It follows from the definition \eqref{v16} of the set $W^\epsilon_a$
that the path $\gamma$ can be chosen in such a way that $F (\bs x_j)
\le H_i - \epsilon$. The previous sum is thus bounded above by $C_0\,
Z_N \, N\, \exp N \{H_i - \epsilon\}$, where $N$ has been
introduced to take care of the length of the path. This estimate is
uniform over $\bs y\in \ms E^a_N$. To conclude the proof of the lemma,
it remains to recall the assertion of Theorem \ref{vs04}.
\end{proof}

\begin{lemma}
\label{vs11}
Under the hypotheses of Proposition \ref{vs09}, for every $a$, $b \in
S_m$,
\begin{equation*}
\lim_{N\to \infty} \beta_m \, r^m_N(a,b) \;=\; \bs r_m(a,b)\;, 
\end{equation*}
where the rates $\bs r_m(a,b)$ are given by \eqref{v23}. 
\end{lemma}

\begin{proof}
By \cite[Lemma 6.8]{bl2}, 
\begin{equation*}
\begin{split}
r^m_N(a,b) \;=\; \frac{1}{2} \, \frac{1}{\mu_N(\ms E^a_N)} \,
\Big\{ & \Cap_{N} (\ms E^a_N, \ms E_{m,N} \setminus \ms E^a_N) \; +\; 
\Cap_{N}(\ms E^b_N, \ms E_{m,N} \setminus \ms E^b_N) \\
&\qquad - \; \Cap_{N} (\ms E^a_N\cup \ms E^b_N, 
\ms E_{m,N} \setminus [\ms E^a_N\cup \ms E^b_N])\, \Big\} \,. 
\end{split}
\end{equation*}
By \eqref{v24},
\begin{equation*}
\mu_N(\ms E^a_N) \;=\; [1+o_N(1)]\, \frac{(2\pi N)^{d/2}}{Z_N}\,
\frac{e^{- N H_i}}{2 \pi N} \, \beta_m \, \bs \mu (a) \;. 
\end{equation*}
The assertion of the lemma follows from this equation, Theorem
\ref{vs04} and the definition of $\bs c_m$ given just above
\eqref{v23}.
\end{proof}

We conclude this section with a calculation which provides an
estimation for the measure of the wells.  Denote by $\bs m^1, \dots,
\bs m^r$ the global minima of $F$ on $\Xi$. We claim that
\begin{equation}
\label{v06}
\lim_{N\to\infty} \frac{e^{N F(\bs m^1)}}{(2\pi N)^{d/2}} \, Z_N \;=\;
\sum_{k=1}^r \frac 1{\sqrt{\det [({\rm Hess}\, F) (\bs m^k)]}}\;\cdot 
\end{equation}
A similar argument yields \eqref{v24}.

Indeed, fix a sequence $\varepsilon_N$ such that $\lim_{N\to\infty} N
\varepsilon^3_N =0$ and for which $\exp\{- N \varepsilon^2_N\}$
vanishes faster than any polynomial. Fix $1\le k\le r$ and denote by
$\bs w^1, \dots, \bs w^d$ the eigenvectors of $({\rm Hess}\, F) (\bs
m^k)$ and by $0<\lambda_1 \le \cdots \le \lambda_d$ the
eigenvalues. Consider the neighborhood $B_N$ of $\bs m^k$ defined by
\begin{equation*}
B_N \;=\; \big\{\bs x\in \Xi_N : |(\bs x - \bs m^k) \cdot \bs w^i|
\le \varepsilon_N \,,\, 1\le i\le d\}\;.
\end{equation*}
It follows from the assumptions on $\varepsilon_N$ and on $F$, from a
second-order Taylor expansion of $F$ around $\bs m^k$, and from a
simple calculation that
\begin{equation*}
\lim_{N\to\infty} \frac{e^{N F(\bs m^1)}}{(2\pi N)^{d/2}} \, 
\sum_{\bs x\in B_N}  e^{- N F(\bs x)} \;=\;
\frac 1{\sqrt{\det [({\rm Hess}\, F) (\bs m^k)]}}\;\cdot 
\end{equation*}

Denote by $B^{(2)}_N$ the neighborhood of $\bs m^k$ defined by
\begin{equation*}
B^{(2)}_N \;=\; \big\{\bs x\in \Xi_N : \Vert \bs x - \bs m^k \Vert
\le \lambda_1/(4C_1) \}\;,
\end{equation*}
where $C_1$ is the Lipschitz constant introduced in assumption (H1).
Clearly, on $B^{(2)}_N$, $F(\bs x) - F(\bs m^1) \ge (1/2) \lambda_1
\Vert \bs x - \bs m^k \Vert^2 - C_1 \Vert \bs x - \bs m^k \Vert^3 \ge
(\lambda_1/4) \Vert \bs x - \bs m^k \Vert^2$. Therefore, as $\Vert
\bs x - \bs m^k \Vert^2 \ge \varepsilon^2_N$ on $B_N^c$ and as $N\,
\varepsilon^2_N \to\infty$,
\begin{equation*}
\lim_{N\to\infty} \frac{e^{N F(\bs m^1)}}{(2\pi N)^{d/2}} \, 
\sum_{\bs x\in B^{(2)}_N  \setminus B_N}  e^{- N F(\bs x)} \;=\;
0\;.
\end{equation*}
On the complement of the union of all $B^{(2)}_N$-neighborhoods of the
minima $\bs m^k$, $F(\bs x) - F(\bs m^1) \ge \delta$ for some
$\delta>0$. In particular the contribution to $Z_N$ of the sum over
this set is negligible. Putting together all previous estimates we
obtain \eqref{v06}.

\section{Proof of Theorem \ref{vs20}}
\label{vsec06}

We prove in this section Theorem \ref{vs20}.  Recall the notation
introduced in Subsection \ref{vsec00}.D. Hereafter, $C_0$ represents a
finite constant independent of $N$ which may change from line to
line. We start with some preliminary results.

\begin{lemma}
\label{vs22}
Fix $1\le i \le i_0$, $1\le j\le \ell_i$, and $a\in S = \{1, \dots
\ell^i_j\}$.  Let $\ms B_a = \cup_{\bs z\in \mf S_a} \ms D_{\bs z}$.
For any sequence $\{\bs x_N: N\ge 1\}$, $\bs x_N\in \ms E^a_N$,
\begin{equation*}
\lim_{N\to\infty} \mb P_{\bs x_N} \Big[ H_{\ms D_a} = 
H_{\ms B_a} \Big]  \;=\; 1\;.
\end{equation*}
\end{lemma}

\begin{proof}
By \cite[Lemma 4.3]{bl9},
\begin{equation*}
\mb P_{\bs x_N} [ H_{\ms D_a} < H_{\ms B_a} ]  \;\le\;
\frac{\Cap_N(\bs x_N,\ms D_a\setminus \ms B_a)}
{\Cap_N(\bs x_N,\ms D_a)}\;\cdot
\end{equation*}
Let $V: \Xi_N \to [0,1]$ be the indicator of the set $\ms D_a\setminus
\ms B_a$. By the Dirichlet principle and a straightforward
computation, $\Cap_N(\bs x_N,\ms D_a\setminus \ms B_a) \le D_N(V) \le
C_0 Z^{-1}_N N^d \exp\{ - N (H_i + \delta_N)\}$. On the other hand, it
is not difficult to construct a divergence-free, unitary flow $\Phi$
from $\ms B_a$ to $\bs x_N$, similar to the one presented in the proof
of Lemma \ref{vs10}, such that $\Vert \Phi\Vert^2 \le C_0 Z_N N \exp\{
- N H_i\}$. Therefore, by Thomson's principle, $\Cap_N(\bs x_N,\ms
D_a)^{-1}\le C_0 Z_N N \exp\{ - N H_i\}$, which proves the lemma in
view of the definition of the sequence $\delta_N$.
\end{proof}

Fix $\bs z\in \mf S_a$ and recall that we denote by $\bs v=\bs w^1$,
$\bs w^j$, $2\le j\le d$, a basis of eigenvectors of ${\rm Hess}\,
F(\bs z)$, where $\bs v$ is the one associated to the unique negative
eigenvalue $-\mu$. Let $B_N = B^{\bs z}_N$ be a mesoscopic
neighborhood of $\bs z$:
\begin{equation}
\label{v46}
B_N = \Big \{\bs x \in \Xi_N : 
\vert (\bs x - \bs z) \cdot \bs v \vert  \le a \,\varepsilon_N 
\,,\, \max_{2\le j\le d} 
\vert (\bs x - \bs z) \cdot \bs w^j \vert  \le 
\varepsilon_N \, \Big\}\;,
\end{equation}
where $a= \max\{ 1, \mu^{-1} ( 1 + \sum_{2\le j\le d} \lambda_j)\}$,
and $\varepsilon_N$ is a sequence of positive numbers such that
$N^{-2}\ll \varepsilon_N^4 \ll N^{-3/2}$, $\varepsilon_N^2 \gg
\delta_N$. The sets $D_{\bs z}$, $\bs z\in \mf S_a$, are contained in
$B_N$ because, by \eqref{v45} and \eqref{v52},
\begin{equation}
\label{v51}
\sup_{\bs x \in D_{\bs z}} \Vert \bs x- \bs z\Vert^2 \;\le\;
(4/\lambda_2)\, \delta_N\; .
\end{equation}

Recall from \eqref{v47} the definition of the outer boundary $\partial
B_N$ of $B_N$, and let $\partial_- B_N$, $\partial_+ B_N$ be the
pieces of the outer boundary of $B_N$ defined by
\begin{equation*}
\begin{split}
& \partial_- B_N = \{\bs x \in \partial B_N :  
(\bs x - \bs z) \cdot \bs v < - a\, \varepsilon_N \}\;, \\
&\qquad \partial_+ B_N = \{\bs x \in \partial B_N :  
(\bs x - \bs z) \cdot \bs v >  a\, \varepsilon_N \}\;.
\end{split}
\end{equation*}
A Taylor expansion of $F$ around $\bs z$ shows that 
\begin{equation}
\label{v48}
\max_{\bs x \in \partial_- B_N \cup \partial_+ B_N} F(\bs x) 
\;\le\; H_i \;-\; \frac 12 \, \varepsilon_N^2 \, 
\big(1 + O(\varepsilon_N) \big)\;.
\end{equation}

Denote by $H_N$ the hitting time of the boundary $\partial B_N$, and
by $H^\pm_N$ the hitting time of the sets $\partial_\pm B_N$.

\begin{proposition}
\label{vs23}
For every $\bs z\in \mf S_a$,
\begin{equation*}
\lim_{N\to \infty} \max_{\bs x\in \ms D_{\bs z}} \Big|\, 
\mb P_{\bs x} \big[ H_N = H^\pm_N\big] \,-\, \frac 12 \, \Big|\;=\;0\;.
\end{equation*}
\end{proposition}

\begin{corollary}
\label{vs26}
Let $\{\bs x^c_N : N\ge 1\}$, $c\in S$, be a sequence of points in
$\ms E^c_N$ and let $\hat S = \hat S_N =\{\bs x^c_N : c\in S\}$.  Fix
$a\not = b \in S$ and $\bs z\in \mf S_{a,b}$. Then,
\begin{equation*}
\lim_{N\to \infty} \max_{\bs x\in \ms D_{\bs z}} \Big|\, 
\mb P_{\bs x} \big[ H_{\hat S} = H_{\bs x^c_N} \big] 
\,-\, q(c) \, \Big|\;=\;0\;,
\end{equation*}
where $q(a) = q(b)=1/2$ and $q(c)=0$ for $c\in S\setminus \{a,b\}$.
\end{corollary}

\begin{proof}
Fix $a\not = b \in S$, $c\in S$, $\bs z\in \mf S_{a,b}$ and $\bs x\in
\ms D_{\bs z}$. Since $H_N \le H_{\hat S}$, by the strong Markov
property,
\begin{equation*}
\mb P_{\bs x} \big[ H_{\hat S} = H_{\bs x^c_N} \big]  \;=\;
\mb E_{\bs x} \Big[ \mb P_{X_N(H_N)} 
\big[ H_{\hat S} = H_{\bs x^c_N}\big]\, \Big]\;.
\end{equation*}
By the proposition, the previous expression is equal to
\begin{equation*}
\sum_{\bs y\in \partial_\pm B_N} \mb P_{\bs x} \big[ X_N(H_N) = \bs y \big]
\mb P_{\bs y} \big[ H_{\hat S} = H_{\bs x^c_N}\big] \;+\; R_N(\bs x) \;,
\end{equation*}
where $\lim_{N\to\infty} \max_{\bs x\in\ms D_{\bs z}} |R_N(\bs x)|
=0$. 

Let $\bs x(t)$, $0\le t\le 1$, be a continuous path from $\bs x^a_N$
to $\bs x^b_N$ for which there exists $0<t_0<1$ such that $F(\bs
x(t))<H_i$ for all $t\not = t_0$ and $\bs x(t_0)=\bs z$. Assume that
this path crosses $B_N$ only at $\partial_\pm B_N$ and assume, without
loss of generality, that it crosses $\partial_- B_N$ before $\partial_+
B_N$. In this case, an argument similar to the one presented in the
proof of Lemma \ref{vs22} yields that
\begin{equation*}
\lim_{N\to\infty} \min_{\bs y\in \partial_+ B_N} \mb P_{\bs y} \big[ H_{\hat S} =
H_{\bs x^b_N}\big] \;=\; 1\;, \quad
\lim_{N\to\infty} \min_{\bs y\in \partial_- B_N} \mb P_{\bs y} \big[ H_{\hat S} =
H_{\bs x^a_N}\big] \;=\; 1\;.
\end{equation*}
In the proof of this assertion, instead of using an indicator function
to bound from above the capacity, as we did in the proof of Lemma
\ref{vs22}, we use the function constructed in Section
\ref{vsec02}. Note also that if the continuous path from $\bs x^a_N$
to $\bs x^b_N$ crosses first $\partial_+ B_N$ and then $\partial_-
B_N$, one has to interchange $a$ and $b$ in the previous displayed
formula.

Up to this point we showed that
\begin{equation*}
\mb P_{\bs x} \big[ H_{\hat S} = H_{\bs x^c_N} \big]  \;=\; 
\mb P_{\bs x} \big[ H_N = H^-_N \big] \mb 1\{c=a\} \;+\; 
\mb P_{\bs x} \big[ H_N = H^+_N \big] \mb 1\{c=b\} \;+\; R'_N(\bs x) \,,
\end{equation*}
where $R'_N(\bs x)$ is a new sequence with the same properties as the
previous one.  To complete the proof it remains to recall the
statement of the proposition.
\end{proof}

The proof of Proposition \ref{vs23} is based on the fact that in a
neighborhood of radius $N^{-1/2}$ around a saddle point $\bs z$ the
re-scaled chain $\sqrt{N} X_N(tN)$ behaves as a diffusion. More
precisely, let $g: \bb R^d \to\bb R$ be a three times continuously
differentiable function and let $G(\bs x) = g (\sqrt{N} (\bs x - \bs
z) \cdot \bs w^1, \dots, \sqrt{N} (\bs x - \bs z) \cdot \bs w^d)$. A
Taylor expansion of the potential $F$ around $\bs z$ gives that for
$\bs x\in B_N$,
\begin{equation}
\label{v49}
(L_N G)(\bs x) \;=\; \frac 1N \sum_{j=1}^d \Big\{ - \lambda_j 
\bs u_j \, (\partial_{\bs x_j} g) (\bs u) + (\partial^2_{\bs x_j}
g)(\bs u) \Big\} \;+\; \frac{R_N}N \;,
\end{equation}
where $\bs u_j = \sqrt{N} (\bs x - \bs z) \cdot \bs w^j$, and $R_N$ is
an error term satisfying 
\begin{equation*}
|R_N| \;\le\; C_0 \, N \varepsilon^2_N \, \Big\{ \frac{C_1(g)}{\sqrt{N}}
\;+\; \frac{C_2(g)}{N} \Big\} \;+\; C_0 \,
\frac{C_3(g)}{\sqrt{N}}\;.
\end{equation*}
In this formula, $C_1(g) = \max_{1\le j\le d}$ $\sup_{\bs u, \Vert \bs
  u \Vert \le a \sqrt{N} \varepsilon_N} |(\partial_{\bs x_j} g)(\bs
u)|$, with a similar definition for $C_2(g)$ and $C_3(g)$, replacing
first derivates by second and thirds.

Identity \eqref{v49} asserts that the process $(\sqrt{N} (X_N(tN) -
\bs z) \cdot \bs w^1, \dots, \sqrt{N} (X_N(tN) - \bs z) \cdot \bs
w^d)$ is close to a diffusion whose coordinates evolve
independently. The first coordinate has a drift towards $\pm \infty$
proportional to its distance to the origin, while the other coordinates
are Ornstein-Uhlenbeck processes.

\begin{lemma}
\label{vs24} 
There exists a finite constant $C_0$ such that for every $\bs z\in \mf
S_a$,
\begin{equation*}
\max_{\bs x\in \ms D_{\bs z}} \mb E_{\bs x} \big[ H_N \big] \;\le\;
C_0 N^{3/2} \varepsilon_N\;.
\end{equation*}
\end{lemma}

\begin{proof}
Let $g:\bb R \to \bb R$ be given by $g(x) = \int_0^x \exp\{-\mu
y^2/2\} \int_0^y \exp\{\mu z^2/2\} dzdy$. It is clear that $g$ solves
the differential equation $\mu x g'(x) + g''(x) = 1$, $x\in \bb
R$. By Dynkin's formula, for every $t>0$, $\bs x\in \ms
D_{\bs z}$,
\begin{equation}
\label{v50}
\mb E_{\bs x} \Big[ G\big(X_N(t\wedge H_N)\big) - G(\bs x) 
\,-\, \int_0^{t\wedge H_N} (L_N G)(X_N(s)) \, ds \Big]\;=\; 0\;,
\end{equation}
where $G(\bs x) = g (N^{1/2} [\bs x - \bs z]\cdot \bs v)$. By \eqref{v49}
and since $|g'(x)|\le C_0$, $|g''(x)|\le C_0 |x|$, $|g'''(x)|\le C_0
x^2$, on $B_N$, $N(L_N G)(x) -1$ is absolutely bounded by $C_0 \sqrt{N}
\varepsilon_N^2$. Therefore,
\begin{equation*}
\big(1 - C_0 \sqrt{N} \varepsilon_N^2 \big) \mb E_{\bs x} \big[ t\wedge
H_N \big] \;\le\; N \,\mb E_{\bs x} \Big[ G\big(X_N(t\wedge H_N)\big)\Big] \;. 
\end{equation*}
Since $|g(x)| \le C_0 |x|$, $\sup_{\bs x \in B_N} |G(\bs x)| \le C_0
\sqrt{N} \varepsilon_N$. To complete the proof of the lemma, it
remains to observe that $\sqrt{N} \varepsilon_N^2 \to 0$ and to let
$t\uparrow\infty$.
\end{proof}

\begin{lemma}
\label{vs25} 
For every $\bs z\in \mf S_a$,
\begin{equation*}
\lim_{N\to\infty} \max_{\bs x\in \ms D_{\bs z}} \mb P_{\bs x} \big[ H_N < H^+_N \wedge
H^-_N \big] \;=\; 0\;.
\end{equation*}
\end{lemma}

\begin{proof}
The proof is similar to the one of the previous lemma. Fix $2\le j\le
d$ and let $g:\bb R \to \bb R$ be given by $g(x) = x^2$.  By Dynkin's
formula, for every $t>0$, $\bs x\in \ms D_{\bs z}$, \eqref{v50} holds
with $G(\bs x) = g (N^{1/2} [\bs x - \bs z]\cdot \bs w^j)$. By
\eqref{v49}, $(L_N G)(\bs x) \le (2 + R_N)/N$ and $R_N/N \le C_0
\varepsilon_N^3$. Therefore, letting $t\uparrow\infty$, by Lemma
\ref{vs24} we get that
\begin{equation*}
\begin{split}
\mb E_{\bs x} \Big[ G\big(X_N(H_N)\big)\Big] \; &\le\; G(\bs x) \;+\;
\Big( \frac 2 N + C_0 \varepsilon_N^3\Big) \, \mb E_{\bs x} \big[ H_N \big] 
\\ &\le\; G(\bs x) \;+\; C_0 \Big(\frac 1N + C_0 \varepsilon_N^3\Big)
N^{3/2} \varepsilon_N \;. 
\end{split}
\end{equation*}

The event $\ms A_N = \{|\, (X_N(H_N)-\bs z)\cdot \bs w^j\, | > 
\varepsilon_N\}$ corresponds to the event that the process $X_N(t)$
reaches the boundary of $B_N$ by hitting the set $\{\bs x\in B_N :
[\bs x-\bs z]\cdot \bs w^j = \pm \varepsilon_N\}$. On this event the
function $G$ is equal to $N \varepsilon_N^2$. Since $G$ is
nonnegative,
\begin{equation*}
N \varepsilon_N^2 \, \mb P_{\bs x} [ \ms A_N ] \;\le\;
\mb E_{\bs x} \Big[ G\big(X_N(H_N)\big)\Big] \;. 
\end{equation*}
On the other hand, by Schwarz inequality and by \eqref{v51}, on the
set $\ms D_{\bs z}$, $G(\bs x)$ is absolutely bounded by $C_0 N \delta_N$.
Putting together the previous two estimates, we get that
\begin{equation*}
\max_{\bs x\in \ms D_{\bs z}} \mb P_{\bs x} [ \ms A_N ] \;\le\;
C_0 \Big( \frac{\delta_N}{\varepsilon_N^2} \;+\; 
\frac 1{\sqrt{N} \, \varepsilon_N} \;+\; \sqrt{N} \,
\varepsilon^2_N\Big)\;.  
\end{equation*}
This completes the proof of the lemma in view of the definition of the
sequence $\varepsilon_N$.
\end{proof}

\begin {proof}[Proof of Proposition \ref{vs23}]
The proof is similar to the one of the two previous lemmas. Let $g(x)
= \int_0^x \exp\{-\mu y^2/2\} dy$. By Dynkin's formula, for every
$t>0$, $\bs x\in \ms D_{\bs z}$, \eqref{v50} holds for $G(\bs x) = g
(N^{1/2} [\bs x - \bs z]\cdot \bs v)$. Since $g''(x) + \mu x g'(x)=0$,
and since the first three derivative of $g$ are uniformly bounded, by
Lemma \ref{vs24} and by \eqref{v49},
\begin{equation*}
\max_{\bs x\in \ms D_{\bs z}} \Big|\,
\mb E_{\bs x} \big[ G\big(X_N(H_N)\big) \big] - G(\bs x)
\,\Big| \;\le\; C_0 \,\varepsilon^3_N \, N \,\to\, 0 \;.
\end{equation*}
On $D_{\bs z}$, the function $G$ vanishes. On the other hand, on the
event $\{H_N = H^\pm_N\}$, $G(X_N(H_N)) = \pm (2\pi/\mu)^{1/2} +
o_N(1)$. Therefore, by Lemma \ref{vs25},
\begin{equation*}
\lim_{N\to\infty} \max_{\bs x\in \ms D_{\bs z}} \Big|\,
\mb P_{\bs x} \big[ H_N = H^+_N \big] \,-\, 
\mb P_{\bs x} \big[ H_N = H^-_N \big] \,\Big| \;=\;0\;.
\end{equation*}
This completes the proof of the proposition in view of Lemma
\ref{vs25}.  
\end{proof}

\begin{proof}[Proof of Theorem \ref{vs20}]
Fix $1\le i \le i_0$, $1\le j\le \ell_i$.  For each $a\in S$, let
$\{\bs x^N_a : N\ge 1\}$ be a sequence of points in $\ms E^a_N$.
Denote by $\hat R_N(a,b)$, $a \not = b\in S$, the jump rates of the
trace of $X_N(t)$ on the set $\{\bs x^a_N : a\in S\}$. By \cite[Lemma
6.8]{bl2},
\begin{equation*}
\begin{split}
\mu_N(\bs x^a_N)\, \hat R_N(a,b) \;=\; \frac 12 \Big\{ & \Cap_N (\{\bs x^a_N\},
\hat S\setminus \{\bs x^a_N\} ) \;+\; \Cap_N (\{\bs x^b_N\},
\hat S\setminus \{\bs x^b_N\} ) \\
&\quad \;-\; \Cap_N (\{\bs x^a_N , \bs x^b_N\},
\hat S\setminus \{\bs x^a_N , \bs x^b_N\} ) \Big\} \;,
\end{split}
\end{equation*}
where $\hat S = \{\bs x^c_N : c\in S\}$. By Remark \ref{vs21}, equation
\eqref{v41}, and the fact that $\bs c_1(a',b') = \bs c(a',b')$, for
$a\not = b\in S$,
\begin{equation*}
\lim_{N\to\infty} \frac{\hat R_N(a,b)}
{\sum_{c\in S, c\not = a} \hat R_N(a,c)} \;=\; p (a,b) \;,
\end{equation*}
where $p(a,b)$ has been introduced in \eqref{v42}.

On the other hand, by \cite[Proposition 6.1]{bl2}, for $a\not = b\in
S$, 
\begin{equation*}
\hat R_N(a,b) \;=\; \lambda_N(\bs x^a_N) \, \mb P_{\bs x^a_N}\big[H_{\bs x^b_N} =
H^+_{\hat S} \big] \;, 
\end{equation*}
and by the strong Markov property, 
\begin{equation*}
\mb P_{\bs x^a_N}[H_{\bs x^b_N}  = H^+_{\hat S} ] \;=\; \mb P_{\bs x^a_N}[  H_{\bs x^b_N}
= H_{\hat S \setminus \{\bs x^a_N\}} ] \, \mb P_{\bs x^a_N}[H^+_{\hat S} <
H^+_{\bs x^a_N}]\;. 
\end{equation*}
It follows from the last three displayed equations that
\begin{equation}
\label{v43}
\lim_{N\to\infty}  \mb P_{\bs x^a_N}[  H_{\bs x^b_N}
= H_{\hat S \setminus \{\bs x^a_N\}} ]  \;=\; p (a,b) \;,
\quad a \,\not =\,  b\in S \;.
\end{equation}

Since any continuous path from $\bs x^a_N$ to $\hat S \setminus \{\bs
x^a_N\}$ must cross $D_a$, $H_{\ms D_a} < H_{\hat S \setminus \{\bs
  x^a_N\}}$ $\mb P_{\bs x^a_N}$-almost surely. Hence, by the strong
Markov property,
\begin{equation*}
\mb P_{\bs x^a_N} [  H_{\bs x^b_N} = H_{\hat S \setminus \{\bs x^a_N\}} ]  \;=\;
\mb E_{\bs x^a_N} \Big[ \mb P_{X_N(H_{\ms D_a})} \big [ H_{\bs x^b_N} = 
H_{\hat S \setminus \{\bs x^a_N\}} \big ]\, \Big ]\;.
\end{equation*}
By Lemma \ref{vs22}, 
\begin{equation}
\label{v53}
\lim_{N\to\infty} \mb P_{\bs x^a_N} [ H_{\ms D_a} < H_{\ms B_a} ] \;=\; 0 \;.
\end{equation}
Therefore, 
\begin{equation*}
\begin{split}
& \lim_{N\to\infty} \mb P_{\bs x^a_N} [  H_{\bs x^b_N} = H_{\hat S \setminus
  \{\bs x^a_N\}} ]  \\
&\qquad =\;
\lim_{N\to\infty} \sum_{\bs z\in \mf S_a} \sum_{\bs y \in \ms D_{\bs z}} 
\mb P_{\bs x^a_N} \big[ H_{\ms D_a} = H_{\bs y} \big ]\, 
\mb P_{\bs y} \big [ H_{\bs x^b_N} = H_{\hat S \setminus \{\bs x^a_N\}} \big ]\;.
\end{split}
\end{equation*}
By the strong Markov property at time $H_{\hat S}$, the previous
expression is equal to
\begin{equation*}
\lim_{N\to\infty} \sum_{\bs z\in \mf S_a} \sum_{\bs y \in \ms D_{\bs
    z}} \sum_{c\in S}
\mb P_{\bs x^a_N} \big[ H_{\ms D_a} = H_{\bs y} \big ]\, 
\mb P_{\bs y} \big [ H_{\bs x^c_N} = H_{\hat S} \big ] \, 
\mb P_{\bs x^c_N} \big [ H_{\bs x^b_N} = H_{\hat S \setminus \{\bs
  x^a_N\}} \big ]\;.
\end{equation*}
By Corollary \ref{vs26}, this limit is equal to
\begin{equation*}
\begin{split}
& \frac 12 \lim_{N\to\infty} \mb P_{\bs x^a_N} \big [ H_{\bs x^b_N} = H_{\hat S \setminus \{\bs
  x^a_N\}} \big ] \, \mb P_{\bs x^a_N} \big[ H_{\ms D_a} = H_{\ms B_a}
\big ] \\
&\quad +\; \frac 12 \lim_{N\to\infty}  \sum_{\bs z\in \mf S_{a,b}} 
\mb P_{\bs x^a_N} \big[ H_{\ms D_a} = H_{\ms D_{\bs z}} \big ]\;.
\end{split}
\end{equation*}
By \eqref{v53}, we may replace in the first line $\mb P_{\bs x^a_N} [
H_{\ms D_a} = H_{\ms B_a}]$ by $1$.

In conclusion, in view of \eqref{v43}, we have shown that
\begin{equation*}
p(a,b)\;=\; \lim_{N\to\infty} \mb P_{\bs x^a_N} [  H_{\bs x^b_N} = H_{\hat S \setminus
  \{\bs x^a_N\}} ]  \;=\; \lim_{N\to\infty}  \sum_{\bs z\in \mf S_{a,b}} 
\mb P_{\bs x^a_N} \big[ H_{\ms D_a} = H_{\ms D_{\bs z}} \big ]\;.
\end{equation*}
This completes the proof of the theorem in the case where the set $\mf
S_{a,b}$ is a singleton. It is not difficult to modify this argument
to handle the case with more than one saddle point between two
wells. Indeed, since the proof does not depend on the behavior of the
function $F$ on $W^c_a$, we can modify $F$ on $W^c_a \setminus
[\cup_{\bs z\in \mf S_a} B_\epsilon(\bs z)]$, for some $\epsilon>0$,
creating new wells of height $H_i$, and turning each saddle point $\bs
z\in \mf S_a$ the unique saddle point between the well $W_a$ and new
well $W'_{\bs z}$.
\end{proof}


\begin{thebibliography}{99}

\bibitem{ag1} L. Avena, A. Gaudilli\`ere: On some random forests with
  determinantal roots. arXiv:1310.1723v3

\bibitem{bl2} J. Beltr\'an, C. Landim: Tunneling and metastability of
  continuous time Markov chains. J. Stat. Phys. {\bf 140}, 1065--1114,
  (2010).

\bibitem{bl7} J. Beltr\'an, C. Landim; Tunneling and metastability of
  continuous time Markov chains II. J. Stat. Phys.  {\bf 149},
  598--618 (2012).

\bibitem{bl9} J. Beltr\'an, C. Landim: A Martingale approach to
  metastability. To appear in Probab. Theory Relat. Fields (2014). 

\bibitem{bbi1} A. Bianchi, A. Bovier, D. Ioffe: Sharp asymptotics for
  metastability in the random field Curie-Weiss
  model. Electron. J. Probab. {\bf 14}, 1541--1603, (2009).

\bibitem{bbi2} A. Bianchi, A. Bovier, D. Ioffe: Pointwise estimates
  and exponential laws in metastable systems via coupling
  methods. Ann. Probab. {\bf 40}, 339--371, (2012).

\bibitem{begk1} A. Bovier, M. Eckhoff, V. Gayrard, M. Klein:
  Metastability in stochastic dynamics of disordered mean field
  models. Probab. Theory Relat. Fields {\bf 119}, 99-161 (2001).

\bibitem{begk2} A. Bovier, M. Eckhoff, V. Gayrard, M. Klein:
  Metastability and low-lying spectra in reversible Markov
  chains. Comm. Math. Phys. {\bf 228}, 219--255 (2002).

\bibitem{begk3} A. Bovier, M. Eckhoff, V. Gayrard, M. Klein:
  Metastability in reversible diffusion processes. I. Sharp
  asymptotics for capacities and exit times. J. Eur. Math. Soc. 
  {\bf 6}, 399--424 (2004).

\bibitem{bgk1} A. Bovier, V. Gayrard, M. Klein: Metastability in
  reversible diffusion processes. II. Precise asymptotics for small
  eigenvalues. J. Eur. Math. Soc. {\bf 7}, 69--99 (2005).

\bibitem{ce1} M. Cameron, E. Vanden-Eijnden: Flows in Complex
  Networks: Theory, Algorithms, and Application to Lennard–Jones
  Cluster Rearrangement. J. Stat. Phys. {\bf 156}, 427--454 (2014).

\bibitem{cgov1}  M. Cassandro, A. Galves, E. Olivieri, M. E. Vares.
  Metastable behavior of stochastic dynamics: A pathwise approach.
  J. Stat. Phys. {\bf 35}, 603--634 (1984).

\bibitem{ee1} W. E, E. Vanden-Eijnden: Towards a theory of transition
  paths. J. Stat. Phys. {\bf 123}, 503--523 (2006).

\bibitem{fw1} M. I. Freidlin, A. D. Wentzell: Random perturbations of
  dynamical systems. Translated from the 1979 Russian original by
  Joseph Sz\"ucs. Second edition. Grundlehren der Mathematischen
  Wissenschaften [Fundamental Principles of Mathematical Sciences],
  260. Springer-Verlag, New York, 1998. 

\bibitem{gov1} A. Galves, E. Olivieri, M. E. Vares: Metastability for
  a Class of Dynamical Systems Subject to Small Random Perturbations.
  Ann. Probab. {\bf 15}, 1288--1305 (1987).

\bibitem{g1} A. Gaudilli\`{e}re. Condenser physics applied to Markov
  chains: A brief introduction to potential theory. Online available
  at http://arxiv.org/abs/0901.3053.

\bibitem{kra1} H. A. Kramers: Brownian motion in a field of force and
  the diffusion model of chemical reactions. Physica {\bf 7}, 284--304
  (1940)

\bibitem{l2} C. Landim: A topology for limits of Markov chains.  To
  appear in Stochastic Process. Appl. arXiv:1310.3646 (2013).

\bibitem{mse1} P. Metzner, Ch. Sch\"utte, E. Vanden-Eijnden: Transition
  path theory for Markov jump processes. SIAM Multiscale
  Model. Simul. {\bf 7}, 1192--1219 (2009).

\bibitem{nwpp1} F. No\'e H. Wu, J. H. Prinz, N.  Plattner: Projected
  and Hidden Markov Models for calculating kinetics and metastable
  states of complex molecules. arxiv 1309.3220v1

\end{thebibliography}
\end{document}